\documentclass{amsart}
\usepackage{amssymb,amstext,amsmath,amscd,amsthm,amsfonts,enumerate,graphicx,latexsym,stmaryrd,multicol,bm}

\usepackage{hyperref}
\usepackage{cleveref}  
\hypersetup{colorlinks=true} 

\usepackage[usenames]{color}
\usepackage[all]{xy}

\newtheorem{thm}{Theorem}[section]
\newtheorem{lem}[thm]{Lemma}
\newtheorem{prop}[thm]{Proposition}
\newtheorem{cor}[thm]{Corollary}
\theoremstyle{definition}
\newtheorem{dfn}[thm]{Definition}

\newtheorem{eg}[thm]{Example}

\newtheorem{rmk}[thm]{Remark}
\theoremstyle{remark}

\newtheorem*{ac}{Acknowlegments}
\newtheorem*{conv}{Convention}

\numberwithin{equation}{thm}
\def\add{\operatorname{add}}

\def\ann{\operatorname{ann}}
\def\Ass{\operatorname{Ass}}
\def\Assh{\operatorname{Assh}}
\def\C{\mathcal{C}}
\def\ca{\operatorname{ca}}
\def\cone{\operatorname{cone}}

\def\D{\mathcal{D}}

\def\Db{\operatorname{D^b}}
\def\depth{\operatorname{depth}}
\def\dim{\operatorname{dim}}
\def\Dsg{\operatorname{D_{sg}}}
\def\Ext{\operatorname{Ext}}
\def\End{\operatorname{End}}
\def\filt{\operatorname{filt}}

\def\grade{\operatorname{grade}}
\def\ge{\geqslant}
\def\height{\operatorname{ht}}
\def\h{\mathrm{H}}
\def\Hom{\operatorname{Hom}}
\def\inf{\operatorname{inf}}

\def\jac{\operatorname{jac}}
\def\k{\mathrm{K}}
\def\le{\leqslant}

\def\m{\mathfrak{m}}
\def\n{\mathfrak{n}}

\def\min{\operatorname{min}}
\def\ipd{\operatorname{IPD}}
\def\Min{\operatorname{Min}}
\def\mod{\operatorname{mod}}

\def\n{\mathfrak{n}}
\def\NF{\mathrm{NF}}
\def\nil{\operatorname{nil}}
\def\sw{\operatorname{SW}}

\def\p{\mathfrak{p}}
\def\pd{\operatorname{pd}}

\def\q{\mathfrak{q}}
\def\r{\operatorname{r}}
\def\rad{\operatorname{rad}}
\def\radius{\operatorname{radius}}
\def\red{\operatorname{red}}
\def\res{\operatorname{res}}
\def\Rfd{\operatorname{Rfd}}
\def\S{\mathcal{S}}
\def\Sing{\operatorname{Sing}}
\def\soc{\operatorname{soc}}
\def\spec{\operatorname{Spec}}
\def\sup{\operatorname{sup}}

\def\syz{\Omega}
\def\T{\mathcal{T}}

\def\thick{\operatorname{thick}}

\def\V{\mathrm{V}}
\def\X{\mathcal{X}}
\def\A{\mathcal{A}}

\def\xx{\bm{x}}
\def\Y{\mathcal{Y}}
\def\Z{\mathcal{Z}}

\begin{document}
\allowdisplaybreaks
\title[Upper bounds for dimensions of singularity categories and their annihilators]{Upper bounds for dimensions of singularity categories and their annihilators}

\author[Souvik Dey]{Souvik Dey}
\address{Department of Algebra, Faculty of Mathematics and Physics, Charles University in Prague, Sokolovska´ 83, 186 75 Praha, Czech Republic}
\email{souvik.dey@matfyz.cuni.cz}
\urladdr{\url{https://orcid.org/0000-0001-8265-3301}}

\author{Yuki Mifune}
\address{Graduate School of Mathematics, Nagoya University, Furocho, Chikusaku, Nagoya 464-8602, Japan}
\email{yuki.mifune.c9@math.nagoya-u.ac.jp}
\begin{abstract}
Let $R$ be a commutative noetherian ring.
Denote by $\mod R$ the category of finitely generated $R$-modules and by $\Db(R)$ the bounded derived category of $\mod R$. 
In this paper, we first investigate localizations and annihilators of Verdier quotients of $\Db(R)$. 
After that, we explore upper bounds for the dimension of the singularity category of $R$ and its (strong) generators.
We extend a theorem of Liu to the case where $R$ is neither an isolated singularity nor even a local ring. Some of our results are more generally stated in terms of Spanier--Whitehead category of a resolving subcategory. 
\end{abstract}
\maketitle
\section{Introduction}
Let $R$ be a commutative noetherian ring.
Denote by $\mod R$ the category of finitely generated $R$-modules and by $\Db(R)$ the bounded derived category of $\mod R$. 

The notion of a localization of an $R$-linear triangulated category at a prime ideal of $R$ is introduced in \cite{BensonIyengarKrause} and studied in \cite{Liu, Matsui} for instance.
The notion of the singularity category $\Dsg(R)$ of $R$ has been introduced by Buchweitz \cite{B}, defined as the Verdier quotient of $\Db(R)$ by the thick closure of $R$.
On the other hand, the concept of the dimension of a triangulated category is introduced by Rouquier \cite{Rouquier}. 
Roughly speaking, this measures the number of extensions necessary to build a given triangulated category from a single object.

Seeking for upper bounds for the dimensions of singularity categories has been studied extensively.
When $R$ is a regular local ring, the singularity category is trivial.
Ballard, Favero, and Katzarkov \cite{BallardFaveroKatzarkov} found an upper bound when $R$ is a local hypersurface with an isolated singularity.
Extending this result, Dao and Takahashi \cite{DaoTakahashi2015} got an upper bound when $R$ is a Cohen--Macaulay local ring with an isolated singularity.
Moreover, Liu \cite{Liu} has extended the results of Dao and Takahashi to the non-Cohen--Macaulay case by using the {\em annihilator ideal} $\ann_{R}\Dsg(R)$ of the singularity category, which is an ideal of $R$ defined to be the set of elements of $R$ annihilating the endomorphism ring of every object of $\Dsg(R)$.
The precise statement of  Liu's theorem is as follows.
\begin{thm}[Liu]\label{***}
Let $(R,\m,k)$ be an equicharacteristic excellent local ring with an isolated singularity. Let $I$ be an $\m$-primary ideal of $R$ contained in $\ann_{R}\Dsg(R)$. Then 
$$
\Dsg(R)={\langle k\rangle}_{(\mu(I)-\depth R+1)\ell\ell(R/I)}.
$$
\end{thm}

In this paper, we extend the framework for obtaining upper bounds for dimensions of singularity categories using computable values to the case of general commutative noetherian rings.
The main result of this paper is the following theorem, which immediately recovers Theorem \ref{***}.
\begin{thm}[\Cref{main thm}] \label{main thm int}

Let $R$ be a commutative noetherian ring and 
$I$ an ideal of $R$ such that $\Sing R\subseteq \V(I)$. Then, $\operatorname{D}_{\operatorname{sg}}(R)=\langle \filt_{\mod R}\{R/\p \mid \p\in \V(I)\} \rangle_{\mu(I)-\grade I+1}$.
If, moreover, $I\subseteq \ann_R \Dsg(R)$, then we have 
\begin{equation*} \label{eq in main thm int}
\Dsg(R) = 
{\langle \mod R/I 
\rangle}_{\mu(I)-\grade I+1}.
\end{equation*}
In particular, the following inequalities hold.
\begin{align}
\dim \Dsg(R) 
& \le (\radius_{\Dsg(R)} (\mod R/I) + 1)(\mu(I)-\grade I+1)-1 \label{ineq1 in main thm int} \\
& \le (\dim\Db(R/I) + 1)(\mu(I)-\grade I+1)-1. \nonumber
\end{align}

\end{thm}

Here, we shall explain the notation used in the above theorem in order from the top.
For a full subcategory $\X$ of $\mod R$, we denote by $\filt_{\mod R}\X$ the smallest full subcategory of $\mod R$ that contains $\X$ and is closed under extensions.
In the right-hand side of the equality for $\Dsg(R)$, we identify $\filt_{\mod R}\{R/\p \mid \p\in \V(I)\}$ and $\mod R/I$ with the full subcategory of $\Dsg(R)$ given as the essential image of the composition of the natural functors $\mod R/I \to \mod R \to \Db(R) \to \Dsg(R)$.
We denote by $\mu(I)$ the minimal number of generators of the ideal $I$.
Finally, in the inequality (\ref{ineq1 in main thm int}), 
$\radius_{\Dsg(R)} (\mod R/I)$ is the least number of extensions to build the full subcategory $\mod R/I$ of $\Dsg(R)$ from a single object in $\Dsg(R)$ (for the precise definition, see Definition \ref{def of dim}).
This notion is a generalization of the (Rouquier) dimension of a triangulated category and introduced in \cite{Oppermann}.

\Cref{main thm int} is a consequence of a more general result about Spanier--Whitehead categories which also generalizes \cite[Corollary 5.8(2)]{bah2}, see \Cref{mthm}.
It is not only a generalization of \Cref{***}, but also has several other applications. Here, let us explain some of them.

If $R$ is a J-$2$ ring, then $\Dsg(R)$ admits a classical generator by a theorem of Iyengar and Takahashi \cite{viet}. Our \Cref{main thm int} refines this statement by providing a sufficient condition for $\Dsg(R)$ to possess a classical generator.
More strongly, if $R$ is a quasi-excellent ring of finite Krull dimension, then the dimension of $\Db(R)$ is finite by a recent theorem of Aoki \cite{Aoki}, and in particular, so is the dimension of $\Dsg(R)$.
As a consequence of \Cref{main thm int}, we can refine the sufficient condition on the finiteness of the dimension of $\Dsg(R)$ as follows.
\begin{cor}[Corollaries \ref{fin of dim} and \ref{mincomp}]\label{fin of dim int}
Let $R$ be a commutative noetherian ring.
\begin{enumerate}[\rm(1)]
\item
Let $I$ be an ideal of $R$ such that $\Sing R\subseteq \V(I)$. If $\Dsg(R/\p)$ has a classical generator for all $\p \in \Min(R/I)$, then $\Dsg(R)$ has a classical generator. In particular, if $R/I$ is a J-$2$ ring, then $\Dsg(R)$ has a classical generator. 
\item
If there exists an ideal $I$ of $R$ contained in $\sqrt{\ann_{R}\Dsg(R)}$ such that $R/I$ is a quasi-excellent ring of finite Krull dimension.
Then the dimension of $\Dsg(R)$ is finite.
\end{enumerate}
\end{cor}

Regarding the hypotheses of both \Cref{main thm int} and \Cref{fin of dim int}, it is to be mentioned that the condition $\Sing R\subseteq \V(I)$ is weaker than $I\subseteq \sqrt{\ann_R \Dsg(R)}$ (see \Cref{ipd}(4)). 


On the other hand, it may not be easy to calculate the ideal $\ann_{R}\Dsg(R)$. 
Thus, we seek to obtain computable but non-trivial ideals contained in $\ann_{R}\Dsg(R)$.
As another consequence of \Cref{main thm int}, we can show that, under mild assumptions, the {\em Jacobian ideal} of $R$, denoted by $\jac(R)$, is such an ideal.
\begin{cor}[\Cref{for calculation}]\label{for calculation int}
Let $R$ be a finitely generated algebra over a field or an equicharacteristic complete noetherian local ring of dimension $d$. 
If $R$ is equidimensional and half Cohen--Macaulay, then one has
\begin{center}
$\Dsg(R) = 
{\langle \mod (R/\jac(R))
\rangle}_{\mu(\jac(R))-\grade (\jac(R))+1}$.
\end{center}
In particular, we have the following inequalities.
\begin{align*}
\dim \Dsg(R)
& \leq (\radius_{\Dsg(R)} (\mod (R/\jac(R))) + 1)(\mu(\jac(R))-\grade (\jac(R))+1)-1 \\
& \leq (\dim\Db(R/\jac(R)) + 1)(\mu(\jac(R))-\grade (\jac(R))+1)-1.
\end{align*}
\end{cor}
Here, a commutative noetherian ring $R$ is said to be {\em half Cohen--Macaulay} if the inequality 
$\dim R_{\p} \le 2\depth R_{\p}$
holds for all prime ideals $\p$ of $R$.
The above result allows us to compute upper bounds for dimensions of singularity categories for a large family of rings, involving the case where $R$ is neither an isolated singularity nor even a local ring.

We also provide some applications and examples by using \Cref{for calculation int}; see Corollaries \ref{cor of dimSing 0}, \ref{cor dimSing 1}, \ref{contable CM} and Examples \ref{dim 41}, \ref{eg dimSing 1}, \ref{uncountable CM}.
In particular, we give an explicit generator and an upper bound for the dimension of the singularity category when the dimension of the singular locus is at most one.

The organization of this paper is as follows.
In section $2$, we state the definitions and basic properties of syzygies of modules, radii of full subcategories of triangulated categories, and $R$-linear triangulated categories for later use.
In section $3$, 
we explore localizations  of Verdier quotients of $R$-linear triangulated categories, and investigate annihilators of Verdier quotients of $\Db(R)$.
In section $4$,
we establish the main result of this paper and deduce many corollaries and their application examples.
\begin{conv}
Throughout this paper, all rings are commutative noetherian rings and all subcategories are full and strict, that is, closed under isomorphism. 
Let $R$ be a commutative noetherian ring.
Denote by $\mod R$ the category of finitely generated $R$-modules, by $\Db(R)$ the bounded derived category of $\mod R$, and by $\Dsg(R)$ the singularity category of $R$.
If $R$ is local, then denote by $\m$ the maximal ideal of $R$ and by $k$ the residue field of $R$.
We write the variables of a polynomial ring or a formal power series ring in capital letters, and denote their images in quotient rings by the corresponding lowercase letters.
\end{conv}

\section{Preliminaries}
In this section, we recall various definitions and basic properties needed later. We begin with the notion of syzygies of modules.
\begin{dfn} \label{def of syz}
Let $M$ be a finitely generated $R$-module.
Take a projective resolution $\cdots \xrightarrow{\delta_{n+1}} P_n \xrightarrow{\delta_n} P_{n-1} \xrightarrow{\delta_{n-1}}\cdots \xrightarrow{\delta_1}P_0 \xrightarrow{\delta_0} M \to 0$ of $M$ with $P_{i} \in \mod R$ for all $i\ge 0$. Then for each $n\ge 0$, the image of $\delta_n$ is called the {\em $n$-th syzygy of $M$} and denoted by $\Omega_R ^n M$. 
Note that $\Omega_R ^n M$ is uniquely defined up to projective summands.
However if $R$ is a local ring, we use {\em minimal free resolution} to define $\Omega_R ^n M$. Hence this is uniquely determined up to isomorphism.
\end{dfn}
\begin{rmk} \label{syz and localization}
Let $R$ be a local ring and
$M$ a finitely generated $R$-module and
$n$ a nonnegative integer. 
Then for all prime ideals $\p$ of $R$, one has
$(\Omega_R ^n M)_{\p}
\cong
\Omega_{R_{\p}} ^n M_{\p} \oplus R^{{\oplus}m}_{\p}$ 
for some $m\ge 0$; see \cite[Proposition 1.1.2]{Avramov}.
\end{rmk}
Next, we recall the notion of the radius of a subcategory of a triangulated category.
This is a generalization of the concept of the dimensions of triangulated categories and introduced in \cite{Oppermann}.
The notation for radius is triangulated analog of the radius of a subcategory of an abelian category introduced by Dao and Takahashi \cite{DaoTakahashiradius}.
Let us begin with recalling the notion of additive closure in an additive category and thick subcategory in a triangulated category.
\begin{dfn} \label{def of dim}
Let $\T$ be an essentially small triangulated category with shift functor $\Sigma$ 
and $\X , \Y$ be subcategories of $\T$.
\begin{enumerate}[\rm(1)]
\item
Let $\mathcal{C}$ be an additive category 
and $\X$ a subcategory of $\mathcal{C}$.
Denote by $\add_{\mathcal{C}} \X$ the {\em additive closure of $\X$}, namely, the subcategory of $\mathcal{C}$ consisting of direct summands of finite direct sums of objects in $\X$. When $\X$ consists of a single object $X$, we simply denote $\add_{\mathcal{C}} \{ X\}$ by $\add_{\mathcal{C}} X$.
\item
A subcategory $\X$ of $\T$ is {\em thick} if 
$\X$ is closed under shifts, cones, and direct summands.
For a subcategory $\X$ of $\T$, we denote by $\thick_{\T} \X$ the {\em thick closure} of $\X$ in $\T$, which is defined as the smallest thick subcategory of $\T$ containing $\X$. We say $\T$ has a {\em classical generator} if $\T=\thick_{\T}(G)$ for some $G\in \T$. 
\item 
We denote by 
$\langle \X \rangle$
the additive closure of the subcategory of $\T$ consisting of objects of the form 
$\Sigma ^{i} X$, where $i \in \mathbb{Z}$ and $X \in \X$.
If $\X$ consists of a single object $X$, we simply denote $\langle \{X\} \rangle$ by $\langle X \rangle$. 
\item
We denote by $\X \ast \Y$ the subcategory of $\T$ consisting of objects $Z$ such that there exists an exact triangle 
$X \to Z \to Y \to \Sigma X$ with $X\in \X$ and $Y\in \Y$. 
We set 
$\X \diamond \Y = \langle \langle \X \rangle \ast \langle \Y \rangle \rangle$. 
\item
For a nonnegative integer $r$, we inductively define the {\em ball of radius $r$ centered at $\X$} as follows.
\begin{equation}
{\langle \X \rangle}_{r} = \nonumber
\begin{cases}
0 & \text{if $r=0$,} \\
\langle \X \rangle & \text{if $r=1$,} \\
{\langle \X \rangle}_{r-1} \diamond \langle \X \rangle & \text{if $r>1$.}
\end{cases}
\end{equation}
If $\X$ consists of a single object $X$, we simply denote ${\langle \{X\} \rangle}_{r}$ by ${\langle X \rangle}_{r}$, and call it {\em the ball of radius $r$ centered at $X$}.
We write ${\langle \X \rangle}_{r} ^{\T}$ when we should specify that $\T$ is the underlying category where the ball is defined.
\item
We define the {\em radius} of $\X$ in $\T$, denoted by $\radius_{\T} \X$, as the infimum of the nonnegative integers $n$ such that there exists a ball of radius $n+1$ centered at an object containing $\X$.
In other words, 
\begin{equation}\label{radius}
\radius_{\T}\X 
= \inf\{n \geq 0\mid \X \subseteq {\langle G \rangle}_{n+1} ^{\T} \text{ for some }G \in \T \}.
\end{equation}
\end{enumerate}
\end{dfn}
\begin{rmk} \label{rmk of dim rad}
Let $\T$ be an essentially small triangulated category and 
$\X, \Y, \Z$ be subcategories of $\T$.
\begin{enumerate}[\rm(1)]
\item
The thick closure
$\thick_{\Db(R)}\{R\}$
is the subcategory of $\Db(R)$ consisting of complexes with finite projective dimension.
\item
Let $M$ be an object in $\T$. 
Then the following conditions are equivalent.
\begin{itemize}
\item 
$M$ belongs to $\X \diamond \Y$.
\item
There exists an exact triangle 
$X \to Z \to Y \to \Sigma X$ 
of objects in $\T$ with
$X \in \langle \X \rangle$, $Y \in \langle \Y \rangle$ 
, and $M$ is a direct summand of $Z$.
\end{itemize}
\item
One has 
$(\X \diamond \Y) \diamond \Z = \X \diamond ( \Y \diamond \Z)$.
\item
There is an ascending chain
\begin{center}
$0 \subseteq \langle \X \rangle \subseteq
{\langle \X \rangle}_{2} \subseteq
\ldots \subseteq
{\langle \X \rangle}_{n}
\subseteq \ldots
\subseteq \thick_{\T} \X$
\end{center}
of subcategories of $\T$, and
the equality $\displaystyle\bigcup_{i\ge 0}{\langle \X \rangle}_{i} = \thick_{\T} \X$ 
holds.
\item
Let $\T'$ be another essentially small triangulated category 
and $F : \T \to \T'$ a triangle functor.
Then for all nonnegative integers $n$, we have
$F({\langle \X \rangle}_{n} ^{\T}) \subseteq
{\langle F(\X) \rangle}_{n} ^{\T'}$. 
Hence the triangle functor
$\thick_{\T} \X \to \thick_{\T'} F(\X)$
is induced by $F$.
\item
The notion of the radius of $\X$ in $\T$ stated in (\Cref{radius}) coincides with what is defined as $\dim_{\T}\X$ in \cite[Definition 2.7]{Oppermann}.
Note that if $\X = \T$, then $\radius_{\T}\T$ is precisely the dimension of $\T$ as a triangulated category introduced by Rouquier \cite{Rouquier}, denoted by $\dim \T$. 
In addition, the inequality
$\dim \T \ge \radius_{\T} \X$ always holds.
\item
Let $I$ be an ideal of $R$.
Then we can consider a triangle functor
$-\otimes_{R/I} ^{\mathbf{L}}R/I : 
\Db(R/I) \to \Db(R)$,
which regards the objects in $\Db(R/I)$ as those in $\Db(R)$.
We often identify $\Db(R/I)$ (or its subcategory) with its essential image under the functor $-\otimes_{R/I} ^{\mathbf{L}}R/I$.

Let $X$ be an object in $\Db(R)$ and
$Y$ an object in $\Db(R/I)$.
Let $Q : \Db(R) \to \Dsg(R)$ be the quotient functor.
Then for all nonnegative integers $n$, we have the following inclusion relations.
\begin{center}
${\langle Y \rangle}_{n} ^{\Db(R/I)} \otimes_{R/I} ^{\mathbf{L}}R/I
\subseteq {\langle Y \rangle}_{n} ^{\Db(R)}$
,\quad
$Q{\langle X \rangle}_{n} ^{\Db(R)}
\subseteq {\langle X \rangle}_{n} ^{\Dsg(R)}$.
\end{center}
Hence, under the identifications stated above, the following inequalities hold.
\begin{equation*}
\xymatrix@C=2pt@R=2pt
{
\dim \Db(R/I) & \ge &  \radius_{\Db(R)}\Db(R/I) & \ge & \radius_{\Dsg(R)}\Db(R/I) \\
\rotatebox{90}{$\le$} & & \rotatebox{90}{$\le$} & &\rotatebox{90}{$\le$}  \\
\radius_{\Db(R/I)}(\mod R/I) & \ge & \radius_{\Db(R)}(\mod R/I) & \ge & \radius_{\Dsg(R)}(\mod R/I).  \\
}
\end{equation*}
In particular, there is an inequality
$\dim \Db(R/I) \ge \radius_{\Dsg(R)}(\mod R/I)$.
\end{enumerate}
\end{rmk}

\begin{dfn} \label{def of ball}
Let $\A$ be an abelian category with enough projectives 
and $\X , \Y$ be subcategories of $\A$.
\begin{enumerate}[\rm(1)]
\item
Let $\mathcal{C}$ be an additive category 
and $\X$ a subcategory of $\mathcal{C}$.
Denote by $\add_{\mathcal{C}} \X$ the {\em additive closure of $\X$}, namely, the subcategory of $\mathcal{C}$ consisting of direct summands of finite direct sums of objects in $\X$. When $\X$ consists of a single object $X$, we simply denote $\add_{\mathcal{C}} \{ X\}$ by $\add_{\mathcal{C}} X$.

\item 
We denote by 
$[\X]$
the additive closure of the subcategory of $\A$ consisting of projective objects of $\A$ and objects of the form 
$\syz^{i}_{\A} X$, where $i \in \mathbb{N}$ and $X \in \X$.
If $\X$ consists of a single object $X$, we simply denote $[\{\X\}]$ by $[\X]$. 
\item
We denote by $\X \bullet \Y$ the subcategory of $\T$ consisting of objects $Z$ such that there exists an exact sequence 
$0\to X \to Z \to Y \to 0$ with $X\in \X$ and $Y\in \Y$. 

\item
For a nonnegative integer $r$, we inductively define the {\em ball of radius $r$ centered at $\X$} as follows.
\begin{equation}
{[ \X ]}_{r} = \nonumber
\begin{cases}
0 & \text{if $r=0$,} \\
[ \X ] & \text{if $r=1$,} \\
{[[\X]_{r-1} \bullet [ \X]]} & \text{if $r>1$.}
\end{cases}
\end{equation}
If $\X$ consists of a single object $X$, we simply denote ${[ \{X\}]}_{r}$ by $[ \X ]_{r}$, and call it {\em the ball of radius $r$ centered at $X$}.
We write $[ \X ]_{r} ^{\A}$ when we should specify that $\A$ is the underlying category where the ball is defined.


\end{enumerate}
\end{dfn}

Here we focus on the inequality
$\dim\Db(R/I) \ge \radius_{\Db(R/I)}(\mod R/I)$.
Combining this with \cite[Proposition 2.6]{AAITY}, we obtain the following inequalities.
\begin{center}
$\radius_{\Db(R/I)}(\mod R/I) \le
\dim\Db(R/I)  \le
2\radius_{\Db(R/I)}(\mod R/I)+1$.
\end{center}
Now we give an example where 
$\radius_{\Db(R/I)}(\mod R/I)$
is smaller than 
$\dim\Db(R/I)$.
\begin{eg}
Let $k$ be a field with characteristic $0$ and we consider 
\begin{center}
$R=k[X, Y]/(X^{2}-Y^{3})$ 
(or $k[\![X, Y]\!]/(X^{2}-Y^{3})$).
\end{center}
Then the {\em Jacobian ideal} of $R$ is 
$I=(x,y^{2})$ (see \Cref{def for lemma} (4) for the definition of a Jacobian ideal).
Hence $R/I \cong k[X]/(X^{2})$ (or $k[\![X]\!]/(X^{2})$)
is the algebra of dual numbers.
Therefore one has $\dim\Db(R/I) = 1$ by \cite[Remark 7.18]{Rouquier}.
On the other hand, the structure theorem for finitely generated modules over a principal ideal domain implies that $\mod R/I$ is the additive closure of $R$ and $k$.
Thus we obtain $\radius_{\Db(R/I)}(\mod R/I) = 0$.
For some more examples, see  \Cref{dim vs rad}.
\end{eg}

We recall the definitions of an $R$-linear triangulated category and its annihilator ideal.
\begin{dfn} \label{def of lin tri cat}
\begin{enumerate}[\rm(1)]
\item
An additive category $\mathcal{C}$ is {\em $R$-linear} if 
for all objects $X\in \C$, there exists a ring homomorphism
$\phi_{X} : R\to \End_{\C}(X)$,
and $\phi$ commutes with the morphisms in $\C$, that is, 
for all morphisms $X \xrightarrow{f} Y$ in $\C$ and 
elements $r \in R$, 
one has $\phi_{Y}(r)\circ f=f\circ \phi_{X}(r)$.
Here we can view $\Hom_{\C}(X, Y)$ as an $R$-module through ring homomorphisms $\phi_{X}$ and $\phi_{Y}$ for all objects $X, Y \in \C$.
\item
A triangulated category $\T$ is {\em $R$-linear triangulated} if $\T$ is $R$-linear as an additive category and for all objects $X, Y$ in $\T$, the isomorphism
$\Hom_{\T}(X,Y) \to \Hom_{\T}(\Sigma X,\Sigma Y)$ 
of abelian groups induced by the shift functor $\Sigma$ is $R$-linear.
\item
Let $\mathcal{C}$ be an  essentially small $R$-linear category. 
For an object $X \in \mathcal{C}$, we define the {\em annihilator} of $X$, 
denoted by $\ann_R X$, as the annihilator ideal of $\End_{\T}(X) = \Hom_{\T}(X,X)$ as an $R$-module.
For a subcategory $\X$ of $\mathcal{C}$, we define
$\ann_R \X$ the intersection of all ideals $\ann_R X$ where $X$ runs thorough objects in $\X$.
\end{enumerate}
\end{dfn}
\begin{rmk}
The bounded derived category of finitely generated $R$-modules $\Db(R)$ and its triangulated subcategories are essentially small $R$-linear triangulated categories. 
In addition, a Verdier quotient of an $R$-linear triangulated category $\T$ by its thick subcategory $\S$ is also an $R$-linear triangulated category via the quotient functor $\T \to \T/\S$.
\end{rmk}
We recall the notion of a resolving subcategory of $\mod R$ and its Spanier--Whitehead category.
\begin{dfn} \begin{enumerate}[\rm(1)]
\item
A subcategory $\X$ of $\mod R$ is {\em resolving} if $\X$ contains $R$ and is closed under direct summands, extensions, and kernels of epimorphisms.

    \item Given a subcategory $\X$ of $\mod R$, $\filt_{\mod R}(\X)$ will denote the smallest extension-closed subcategory of $\mod R$ containing $\X$. Similarly, $\res_{\mod R}(\X)$ will denote the smallest resolving subcategory of $\mod R$ containing $\X$. See \cite[Definition 1.7]{stcm} for details. 


    \item Given a resolving subcategory $\X$ of $\mod R$, $\sw(\X)$ will denote the Spanier--Whitehead category as defined in \cite[Definition 2.6]{bah2}.  $\sw(\X)$ is naturally identified with a full triangulated subcategory of $\Dsg(R)$. Moreover, since $\X$ is closed under direct summands in $\mod R$, and every object of $\sw(\X)$ is of the form $X[m]$ for some $X\in \X, m\in \mathbb Z$, hence it follows from \cite[Lemma 5.2(2)]{bah2} that $\sw(\X)$ is closed under direct summands in $\Dsg(R)$, and thus $\sw(\X)$ is a full thick subcategory of $\Dsg(R)$. In particular, one has $\ann_{R}\End_{\Dsg(R)}(M)=\ann_{R}\End_{\sw(\X)}(M)$ and $\langle M\rangle_n^{\Dsg(R)}=\langle M\rangle_n^{\sw(\X)}$ for every $M\in \sw(\X)$ and $n\ge0$. 
\end{enumerate} 
    
\end{dfn}

\begin{rmk}\label{remsw} \begin{enumerate}[\rm(1)]

\item For a subset $\Phi$ of $\spec R$, the singularity category $\operatorname{D}^\Phi_{\operatorname{sg}}(R)$ supported at $\Phi$ was introduced and studied in \cite[Definition 2.8, Theorem 4.1, and Corollary 4.2]{kos}. It follows from \cite[Theorem 3.2, Definition 3.3(1), and Proposition 4.2]{bah2} and \cite[Corollary 4.2]{kos} that for a specialization-closed subset $\Phi$ of $\spec R$, $\sw(\operatorname{res}_{\mod R}\{R/\p \mid \p\in \Phi\})=\operatorname{D}^\Phi_{\operatorname{sg}}(R)$. In particular, one has $\sw(\mod R)=\Dsg(R)$.

    \item Let $\X$ be a resolving subcategory of $\mod R$. Let $X\in \X,G\in \sw(\X)$ and $n\ge 0$ an integer. We know $G\cong H[m]$ for some $H\in \X$ and $m\in \mathbb Z$. It follows from the argument of \cite[Proposition 5.3]{bah2} that $$X\in \langle G\rangle_n^{\sw(\X)} \iff  \syz^r X\in [H]^{\mod R}_n \text{ for some } r\ge 0 \iff X\in \langle G\rangle_n^{\Dsg(R)}.$$ 

\end{enumerate} 
    
\end{rmk}
We recall some notions concerning certain subsets of the prime spectrum of $R$, denoted by $\spec R$.
\begin{dfn}
\begin{enumerate}[\rm(1)]
\item
A subset $\Phi$ of $\spec R$ is {\em specialization-closed} if $\p\subseteq\q$ are prime ideals of $R$ and $\Phi$ contains $\p$, then $\q$ belongs to $\Phi$.
\item
We denote by $\Sing R$ the {\em singular locus} of $R$, that is, the set of prime ideals $\p$ of $R$ such that the localization $R_{\p}$ is not a regular local ring. 
\item
For a finitely generated $R$-module $M$, we denote by $\NF(M)$ the {\em nonfree locus} of $M$, that is, the set of prime ideals $\p$ of $R$ such that $M_{\p}$ is not free $R_{\p}$-module.
For a subcategory $\X$ of $\mod R$, we set $\NF(\X):=\bigcup_{M\in \X} \NF(M)$.
\item
For a finitely generated $R$-module $M$, we denote by $\ipd(M)$ the {\em infinite projective dimension locus} of $M$, that is, the set of prime ideals $\p$ of $R$ such that $\pd_{R_{\p}}(M_{\p})=\infty$.
For a subcategory $\X$ of $\mod R$, we set $\ipd(\X):=\bigcup_{M\in \X} \ipd(M)$.
\end{enumerate}
\end{dfn}

\begin{rmk}\label{ipd}
Let $\X$ be a subcategory of $\mod R$.
\begin{enumerate}[\rm(1)]
\item
The subsets $\Sing R$, $\NF(\X)$, and $\ipd(\X)$ are specialization-closed subsets of $\spec R$.
\item
It is easy to see that $\ipd(\res_{\mod R}\X)=\ipd(\X)\subseteq \Sing R$. 
\item 
Let $\Phi$ be a subset of $\spec R$. 
Then one has $\Phi\cap\Sing R \subseteq \ipd(\{R/\p \mid \p\in \Phi\})$. if, moreover, $\Phi$ is specialization-closed, then the equality holds.
Indeed, if $\p\in \Phi\cap \Sing R$, then $\p\in \ipd(R/\p)\subseteq \ipd(\{R/\p \mid \p\in \Phi\})$. Now let $\Phi$ be specialization-closed. Then, $\ipd(R/\p)\subseteq \V(\p)\cap \Sing R\subseteq \Phi\cap \Sing R$ for each $\p\in\Phi$. Thus, $\ipd(\{R/\p \mid \p\in \Phi\})\subseteq \Phi\cap \Sing R$. 
\item 
In view of \cite[3.1 and Proposition 4.2]{dey2025annihilationcohomologystronggeneration} we see that $\ca(\X)=\bigcap_{M\in \X}\ann_{R}\End_{\Dsg(R)}(M)=\ann_R \sw(\X)$, see \cite[Definition 5.7(1)]{bah2} for the notion of $\ca(\X)$. 
In addition, one has $\ipd(\X)=\bigcup_{M\in \X} \V(\ann_{R}\End_{\Dsg(R)}(M))\subseteq \V\left(\bigcap_{M\in \X} \sqrt{\ann_{R}\End_{\Dsg(R)}(M)}\right)\subseteq \V(\ca(\X))=\V(\ann_R \sw(\X))$ (see \cite[Propositions 3.10 and 4.2]{dey2025annihilationcohomologystronggeneration}). 
\end{enumerate}
    
\end{rmk}
We end this section by recalling the notion of Koszul objects in an $R$-linear triangulated category.
\begin{dfn} \label{def of koszul obj}
Let $\T$ be an $R$-linear triangulated category.
For an object $X$ in $\T$ 
and a sequence $\xx = x_1, \ldots, x_n$ of elements in $R$, 
we inductively define the {\em Koszul object} of $\xx$ on $X$, denoted by $X/\!\!/\xx$, as follows.
\begin{equation}
X/\!\!/\xx = \nonumber
\begin{cases}
X & \text{if $n=0$,} \\
\cone(x_n 1_{X/\!\!/(x_1, \ldots, x_{n-1})}) & 
\text{if $n>0$.}
\end{cases}
\end{equation}
Note that if $X$ is an object in $\Db(R)$, 
then the Koszul object
$X/\!\!/\xx$ 
is precisely the Koszul complex 
$\k(\xx, X)$
of $\xx$ on $X$.
\end{dfn}
\begin{rmk} \label{end annihilators}
Let $\T$ be an $R$-linear triangulated category.
For an object $X$ in $\T$ 
and a sequence $\xx =x_{1}, \ldots, x_n$ of elements in $\ann_R X$,
we have an isomorphism
$X/\!\!/\xx \cong
\displaystyle\bigoplus_{i=0} ^{n}
(\Sigma^{i}X)^{\oplus \binom{n}{i}}$
in $\T$. This is proved by induction on $n\ge0$. 
In particular, $X$ is a direct summand of $X/\!\!/\xx$ in $\T$.
\end{rmk}

\section{Localizations of $R$-linear triangulated categories}
The main purpose of this section is to prove \Cref{localiz of ver quot} based on that of \cite[Lemma 4.3]{Liu}, and deduce some results.
First of all, we state some general statements about a localization of an $R$-linear triangulated category.

Let $R$ be a commutative noetherian ring and $\p$ a prime ideal of $R$. 
Let $\T$ be an $R$-linear triangulated category
and $\X$ a subcategory of $\T$.
Let $\T'$ be an $R_{\p}$-linear triangulated category
and $F : \T \to \T'$ a triangle $R$-functor.
We define $\S_{\T}(\p)$ to be the subcategory of $\T$ consisting of objects $X$ which satisfy $\End_{\T}(X)_{\p} =0$.
This is equivalent to say that there exists an element $s \in R \setminus \p$ such that $s \cdot 1_{X} = 0$ in $\T$. 
By \cite[Lemma 3.1(a)]{Matsui}, $\S_{\T}(\p)$ is a thick subcategory of $\T$.
We define the {\em localization} of $\T$ at $\p$ as the Verdier quotient $\T_{\p} = \T/\S_{\T}(\p)$.

We can see that
$\S_{\T}(\p) \subseteq 
\{X\in \T \mid FX=0 \text{ in } \T' \}$ holds.
Indeed, let $X$ be an object in $\T$ with 
$\End_{\T}(X)_{\p} =0$. 
Then there exists an element $s \in R \setminus \p$ such that
$s \cdot 1_{X} = 0$.
Hence $X$ is a direct summand of $X/\!\!/s$ in $\T$, so it is enough to show that $F(X/\!\!/s)=0$.
Since $\End_{\T'}(FX)$ possesses an $R_{\p}$-module structure, 
the element $s$ acts unitary on $\End_{\T'}(FX)$. 
Thus we have
$F(X/\!\!/s) 
= F(\cone(s \cdot 1_{X}))
\cong \cone(F(s \cdot 1_{X})) 
= \cone(s \cdot 1_{FX}))
= 0$.
Therefore the functor $F$ induces the triangle functor
$ F_{\p} ^{1} : \T_{\p} \to \T'$.
Consequently, for all objects $X, Y$ in $\T$, we obtain the following commutative diagram.
\begin{equation*} 
\xymatrix{
  & \Hom_{\T}(X,Y)_{\p} \ar[ldd]_{F_{\p} ^{0}(X,Y)} \ar[rdd]^{F(X,Y)_{\p}}   &    \\
  & \Hom_{\T}(X,Y) \ar[u] \ar[ld] \ar[rd] &    \\
\Hom_{\T_{\p}}(X,Y) \ar[rr]^{F_{\p} ^{1}(X,Y)}  &   & \Hom_{\T'}(FX, FY)
}
\end{equation*}
Here the homomorphism 
$F_{\p} ^{0}(X,Y) : \Hom_{\T}(X,Y)_{\p} \to \Hom_{\T_{\p}}(X,Y)$ 
in the above diagram induced by quotient functor 
$\T \to \T_{\p}$ 
is an isomorphism by \cite[Lemma 3.1(c)]{Matsui}. 
As the outer diagram commutes, the functor
$F_{\p} ^{1}$ is a triangle equivalence if and only if
$F$ is dense and $F(X,Y)_{\p}$ is an isomorphism for all $X, Y \in \T$.

An additive functor $F : \C \to \D$ of additive categories is {\em essentially dense} if for each object $Y$ in $\D$, there exists an object $X$ in $\C$ such that $Y$ is a direct summand of $FX$.

In the following proposition, we investigate the relationship between Verdier quotient and localization of an $R$-linear triangulated category.
Consequently, the following (2) and (3) each provide a generalization of \cite[Lemmas 4.1 and 4.3]{Liu}, respectively.
\begin{prop} \label{localiz of ver quot}
Let $\p$ be a prime ideal of $R$, 
$\T$ an $R$-linear triangulated category, and
$\T'$ an $R_{\p}$-linear triangulated category.
Let $F : \T \to \T'$ be a triangle $R$-functor.
Then for all subcategories $\X$ of $\T$, the following hold.
\begin{enumerate}[\rm(1)]
\item 
Assume that $F_{\p} ^{1}$ is full.
Then the induced functor
$F : \thick_{\T} \X \to \thick_{\T'} F\X$
is essentially dense.
\item
Assume that $F_{\p} ^{1}$ is fully faithful.
Let $X$ be an object in $\T$.
If the object $FX$ belongs to $\thick_{\T'} F\X$, then
there exists an element $s \in R \setminus \p$ such that 
$s \cdot 1_{X} = 0$ in $\T/\thick_{\T} \X$.
\item
Assume that $F_{\p} ^{1}$ is a triangle equivalence.
Then the induced functor
$\overline{F} : (\T/\thick_{\T}\X)_{\p} \to \T'/\thick_{\T'} F\X$
of Verdier quotients is also a triangle equivalence.
\end{enumerate}
\end{prop}
\begin{proof}
(1) It is enough to show that for all $n \ge 1$ and objects
$Y$ in ${\langle F\X \rangle}_{n} ^{\T'}$, 
there exists an object $X$ in 
${\langle \X \rangle}_{n} ^{\T}$ such that
$Y$ is a direct summand of $FX$.
We shall prove this by induction on $n\ge 1$.
When $n=1$, the object $Y$ is a direct summand of
$\displaystyle\bigoplus_{i=1} ^{k}(\Sigma^{a_{i}}FX_{i}) \cong F(\displaystyle\bigoplus_{i=1} ^{k}\Sigma^{a_{i}}X_{i})$ for some $k\ge0, a_{i} \in \mathbb{Z}$, and $X_{i} \in \X$. 
As the object $\displaystyle\bigoplus_{i=1} ^{k}(\Sigma^{a_{i}}X_{i})$
belongs to $\langle \X \rangle ^{\T}$, it is verified.

Suppose $n>1$ and $Y$ is in ${\langle F\X \rangle}_{n} ^{\T'}$.
Then there exists an exact triangle 
\begin{center}
$A \to B \to C \to \Sigma A$ 
\end{center}
in $\T'$ with $A \in {\langle F\X \rangle}_{n-1} ^{\T'}$
, $C \in \langle F\X \rangle ^{\T'}$, and $Y$ is a direct summand of $B$.
By the induction hypothesis and the case where $n=1$, 
there exist objects
$Z \in {\langle \X \rangle}_{n-1} ^{\T}$, 
$W \in \langle \X \rangle^{\T}$, and 
$A', C' \in \T'$ such that
$A \oplus A' \cong FZ$, and 
$C \oplus C' \cong FW$ in $\T'$.
By taking the direct sums of $A'$ and $C'$ in the above exact triangle, we have the following exact triangle.
\begin{center}
$FZ \to B\oplus A' \oplus C' \to FW 
\xrightarrow{\delta_{0}} \Sigma (FZ)$.
\end{center}
Let $\delta$ be the composition of morphisms 
$FW \xrightarrow{\delta_{0}} \Sigma (FZ) 
\xrightarrow[\cong]{f} F(\Sigma Z)$.
Since the homomorphism  
$\Hom_{\T}(W, \Sigma Z)_{\p} \to \Hom_{\T'}(FW, F(\Sigma Z))$ 
induced by $F$ is surjective, 
there exist an element $s \in R\setminus \p$ and 
$\eta \in \Hom_{\T}(W, \Sigma Z)$ such  that
$F(\eta/1)=s\delta$.
Hence we get an exact triangle 
\begin{center}
$Z \to V \to W \xrightarrow{\eta} \Sigma Z$ 
\end{center}
in $\T$, where $V$ is a mapping cocone of $\eta$. 
Since $Z$ is in ${\langle \X \rangle}_{n-1} ^{\T}$ and
$W$ is in $\langle \X \rangle^{\T}$, 
the object $V$ belongs to ${\langle \X \rangle}_{n} ^{\T}$.
Applying the functor $F$ to the above exact triangle, we obtain the following diagram where each row is an exact triangle in $\T'$.
\begin{equation*}
\xymatrix{
FZ \ar[r] & FV \ar[r] & FW \ar[r]^{F(\eta)} & F(\Sigma Z) \\
FZ \ar[r] \ar@{=}[u] & B\oplus A'\oplus C' \ar[r] & FW \ar@{=}[u] \ar[r]^{\delta_{0}} & \Sigma(FZ) \ar[u]^{\cong}_{s\cdot f} 
}
\end{equation*}
Thus one has an isomorphism $FV \cong B\oplus A'\oplus C'$.
Since $Y$ is a direct summand of $B$, it is also a direct summand of $FV$. This completes the proof.

(2) Since $FX$ belongs to $\thick_{\T'} F\X$, 
there exists an object $X'$ in $\thick_{\T} \X$ such that 
$FX$ is a direct summand of $FX'$ in $\T'$ by the consequence of (1).
Hence there exist homomorphisms
$FX \xrightarrow{\xi} FX' \xrightarrow{\eta} FX$ 
such that $\eta \xi = 1_{FX}$.
By the surjectivity of 
$\Hom_{\T}(X, X')_{\p} \to \Hom_{\T'}(FX, FX')$
and $\Hom_{\T}(X', X)_{\p} \to \Hom_{\T'}(FX', FX)$, 
we can find elements $s, t \in R\setminus \p$ 
and homomorphisms
$X \xrightarrow{f}X'\xrightarrow{g}X$ in $\T$
satisfying that
$s\cdot \xi = F(f/1)$ and $t\cdot \eta = F(g/1)$.
The condition $\eta \xi = 1_{FX}$ implies that
$F((gf-ts)/1)=0$ in $\End_{\T'}(FX)$.
By the injectivity of $\End_{\T}(X)_{\p} \to \End_{\T'}(FX)$, 
there exists an element $u \in R\setminus \p$ such that
$ugf=uts$ in $\End_{\T}(X)$.
Thus the homomorphism $(uts)\cdot 1_{X}$
factors through the object $X'$ in $\thick_{\T}\X$, and this means that 
$(uts) \cdot 1_{X} = 0$ in $\T/\thick_{\T} \X$.

(3) Since the functor $F$ is dense, the induced functor $\overline{F}$ is also dense. 
Hence we only need to prove that for all objects $X, Y \in \T$, the induced homomorphism
$\Hom_{\T/\thick_{\T}\X}(X, Y)_{\p}
\to
\Hom_{\T'/\thick_{\T'}F\X}(FX, FY)$
is bijective.
This proof is basically the same as the proof in \cite[Lemma 4.3]{Liu}, obtained by replacing $\Db(R), \Db(R_{\p})$, $(-)_{\p}$, and $\{ R\}$ with $\T, \T', F$, and $\X$ respectively.
Therefore we omit the details of the arguments.
\end{proof}
By taking $\Db(R)$, $\Db(R_{\p})$, and $(-)_{\p}$ as $\T$, $\T'$, and $F$ respectively, we obtain the following triangle equivalence, which means that the Verdier quotient and the localization of bounded derived category are commutative.
In particular, by taking $\{R\}$ as $\X$ in the following corollary, we can recover \cite[Corollary 4.4]{Liu}.
\begin{cor} \label{app to derv cat}
Let $\X$ be a subcategory of $\Db(R)$.
Then for all prime ideals $\p$ of $R$, we have the following triangle equivalence.
\begin{center}
$(\Db(R)/\thick \X)_{\p} \cong
\Db(R_{\p})/\thick(\X_{\p})$.
\end{center}
\end{cor}
\begin{proof}
By \cite[Lemma 4.2]{AiharaTakahashi2015} and \cite[Lemma 5.2(b)]{AvramovFoxby1991}, 
the functor $(-)_{\p} : \Db(R) \to \Db(R_{\p})$
is dense and the induced homomorphism
$\Hom_{\Db(R)}(X,Y)_{\p} \to \Hom_{\Db(R_{\p})}(X_{\p},Y_{\p})$
is an isomorphism for all $X,Y \in \Db(R)$.
Hence the induced functor
$\Db(R)_{\p} \to \Db(R_{\p})$ is a triangle equivalence.
Thus by \Cref{localiz of ver quot}(3), 
we obtain the triangle equivalence as desired.
\end{proof}
Consequently, we obtain a characterization of the closed subset defined by the annihilator of a Verdier quotient of bounded derived category of $\mod R$.
By taking $\{R\}$ as $\X$ in the following corollary, we can recover \cite[Theorem 4.6 and Remark 4.5(2)]{Liu}.
\begin{cor} \label{closedness}
Let $\X$ be a subcategory of $\Db(R)$.
\begin{enumerate}[\rm(1)]
\item 
Suppose that the dimension of $\Db(R)/\thick \X$ is finite. 
Then the following equality holds.
\begin{center}
$\{\p \in \spec R \mid \Db(R_{\p}) \neq \thick (\X_{\p}) \}
= \V(\ann_{R}(\Db(R)/\thick \X))$.
\end{center}
In particular, the set of prime ideals $\p$ of $R$ satisfying the equality 
$\Db(R_{\p}) = \thick (\X_{\p})$ is an open subset of $\spec R$.
\item
Let $X$ be an object in $\Db(R)$.
Then we have the following equality.
\begin{center}
$\{\p \in \spec R \mid X_{\p} \notin \thick(\X_{\p}) \}
= \V(\ann_{R}\End_{\Db(R)/\thick\X}(X))$.
\end{center}
\end{enumerate}
\end{cor}
\begin{proof}
Let $\p$ be a prime ideal of $R$.

(1) By \Cref{app to derv cat} and \cite[Proposition 3.11]{Liu}, we have 
$\Db(R_{\p}) \neq \thick (\X_{\p})$
if and only if
$(\Db(R)/\thick \X)_{\p} \neq 0$
if and only if
$\p \in \V(\ann_{R}(\Db(R)/\thick \X))$.

(2) By \Cref{app to derv cat} and \cite[Corollary 3.9]{Liu}, we have 
$X_{\p} \notin \thick (\X_{\p})$
if and only if
$X\neq 0$ in $(\Db(R)/\thick \X)_{\p}$
if and only if
$\p \in \V(\ann_{R}\End_{\Db(R)/\thick\X}(X))$.
\end{proof}
A local ring $(R, \m, k)$ is an {\em isolated singularity} if the inclusion relation 
$\Sing R \subseteq \{\m\}$ holds.

In the following example, we provide some applications of \Cref{closedness} when $R$ is a local ring.
\begin{eg} \label{app of clsdness}
Let $(R, \m, k)$ be a commutative noetherian local ring.
\begin{enumerate}[\rm(1)]
\item 
Suppose that the dimension of $\Db(R)/\thick \{k\}$ is finite. If $R$ is not artinian, then one has
\begin{center}
$\V(\ann_{R}(\Db(R)/\thick \{k\}))=\spec R$.
\end{center}
\item
Suppose that the dimension of $\Db(R)/\thick \{k\oplus R\}$ is finite. If $R$ is not an isolated singularity, then one has
\begin{center}
$\V(\ann_{R}(\Db(R)/\thick \{k\oplus R\}))=\Sing R$.
\end{center}
\item
Let $D$ be a dualizing complex for $R$.
Then the subcategory of $\Db(R)$ consisting of complexes with finite injective dimension coincides with $\thick\{D\}$ as the functor 
$\mathbf{R}\!\Hom_{R}(-,D)$
gives a duality 
$\Db(R) \xrightarrow{\cong} \Db(R)$; see \cite[(2.7)]{AvramovFoxby1997} for instance.
Hence for an object $X$ in $\Db(R)$, the ideal 
$\ann_{R} \End_{\Db(R)/\thick \{D\}}(X)$
defines the infinite injective dimension locus of $X$, that is, the set of prime ideals $\p$ of $R$ such that the $R_{\p}$-complex $X_{\p}$ has infinite injective dimension.

In addition, if the dimension of $\Db(R)/\thick \{D\}$ is finite, then one has
\begin{center}
$\V(\ann_{R}(\Db(R)/\thick \{D\}))=\Sing R$
\end{center}
by \Cref{closedness}(1).
However, the duality 
$\mathbf{R}\!\Hom_{R}(-,D) : \Db(R) \xrightarrow{\cong} \Db(R)$
induces a category equivalence
\begin{center}
$\Dsg(R) \cong \Db(R)/\thick \{D\}$
\end{center}
of Verdier quotients of $\Db(R)$.
Hence we have
\begin{center}
$\End_{\Db(R)/\thick \{D\}}(X) \cong 
\End_{\Dsg(R)}(\mathbf{R}\!\Hom_{R}(X,D)
)$, and \\
$\ann_{R}\Dsg(R) = \ann_{R}(\Db(R)/\thick \{D\})$.
\end{center}
Note that the infinite injective dimension locus of $X$ coincides with the infinite projective dimension locus of $\mathbf{R}\!\Hom_{R}(X,D)$.
Thus, the results stated above essentially follow from \cite[Theorem 4.6 and Remark 4.5(2)]{Liu}.
\item
Let $D$ be a dualizing complex for $R$ and 
$X$ an object in $\Db(R)$.
Then the set of prime ideals $\p$ of $R$ such that
$X_{\p} \notin \thick\{R_{\p}\oplus D_{\p}\}$
is a closed subset of $\spec R$.
In addition, if the dimension of $\Db(R)/\thick\{R\oplus D\}$ is finite, 
then the set of prime ideals $\p$ of $R$ such that
$\Db(R_{\p}) \neq \thick\{R_{\p}\oplus D_{\p}\}$
is a closed subset of $\spec R$.
\end{enumerate}
\end{eg}

\section{On upper bounds for the dimensions of singularity categories}
The main purpose of this section is to state and prove \Cref{mthm}. 
After that we deduce Corollaries \ref{main thm}, \ref{fin of dim}, \ref{new1}, \ref{mincomp}, \ref{new3}, \ref{new2}, \ref{for calculation}, \ref{cor of dimSing 0}, \ref{thm of Liu}, \ref{cor dimSing 1}, and \ref{contable CM}.
Using the above results, we compute upper bounds for the dimensions of the singularity categories over specific rings.

First of all, we recall the notion of the restricted flat dimension of a module.
The finiteness of the restricted flat dimension of a module plays an important role in removing the assumption that $R$ is local in \Cref{main thm}.
\begin{dfn} \label{def of Rfd}
Let $M$ be a finitely generated $R$-module.
\begin{enumerate}[\rm(1)]
\item 
We define the {\em (large) restricted flat dimension} of $M$, 
denoted by $\Rfd_R (M)$, as follows.
\begin{center}
$\Rfd_R (M) 
= \sup \{\depth R_{\p}-\depth M_{\p} \mid
\p \in \spec R \}$. 
\end{center}
By using \cite[Theorem 1.1]{AvramovIyengarLipman} and \cite[Proposition 2.2, Theorem 2.4]{ChristensenFoxbyFrankild}, we have 
$\Rfd_R (M) \in \mathbb{Z}_{\ge 0} \cup \{-\infty \}$. 
Note that $\Rfd_R (M) = -\infty$ if and only if $M = 0$.
\item
We denote by $\mu (M)$ the infimum of nonnegative integers $n$ such that $M$ is generated by $n$ elements as an $R$-module.
If $(R,\m,k)$ is a local ring, then $\mu (M)$ coincides with the dimension of 
$M \otimes_R k$ as a $k$-vector space. 
\end{enumerate}
\end{dfn}
We are ready to prove the main result of this section.

\begin{thm}\label{mthm} Let $\X$ be a resolving subcategory of $\mod R$.  Assume $I$ is an ideal of $R$ such that $\ipd(\X)\subseteq \V(I)$ and $R/\p \in \sw(\X)$ for every $\p \in \V(I)$. Then, $\sw(\X)=\langle \filt_{\mod R}\{R/\p \mid \p\in \V(I)\} \rangle^{\sw(\X)}_{\mu(I)-\grade I+1}$. If, moreover, $I\subseteq \ann_R \sw(\X)$, then $\sw(\X)=\langle \mod R/I \rangle^{\sw(\X)}_{\mu(I)-\grade I+1}$.  
    
\end{thm}

\begin{proof}  We will use \Cref{remsw}(2) without further reference. 

Let 
$I = (\xx)$ 
where $\xx = x_{1}, \ldots, x_{n}$ with $n=\mu(I)$.
Let $M\in \X$.

By hypothesis, $\ipd(M)\subseteq \V(I)$. By \cite[3.1 and Proposition 3.10]{dey2025annihilationcohomologystronggeneration} we have $\sqrt I\subseteq \ann_R \Ext_R^{\ge 1}(\syz^m M, \mod R)$ for some $m\ge \Rfd_R M$. Then, there exists $e\ge 1$ such that $(x_1^e,\ldots,x_n^e)\subseteq \ann_R \Ext_R^{\ge 1}(\syz^m M, \mod R)$. Put $J=(x_1^e,\ldots,x_n^e)$. Then, one has $\sqrt J=\sqrt I$.

Put
$X = \Omega_{R} ^{m}M$. Note that $\grade(I,X)=\grade(J,X)$. 
Let $\k(\xx^e, X)$ be the Koszul complex of $x_1^e,...,x_n^e$ on $X$, and $\h^{i}(\xx^e, X)$ its $i$-th cohomology for each 
$i\in \mathbb{Z}$.
Then the grade sensitivity(\cite[Theorem 1.6.17]{BH}) implies that
\begin{center}
$n-\grade(J,X) = -\inf\{i \mid \h^{i}(\xx^e,X) \neq 0\}$.
\end{center}
Hence for each integer $i$ which is bigger than $0$ or smaller than $\grade(J,X)-n$, 
the cohomology $\h^{i}(\xx^e, X)$ vanishes.
Thus by using \cite[Lemma 2.1(2)]{DaoTakahashi2015}, 
we have 
\begin{align*}
\k(\xx^e, X) \in 
 & {\left\langle \displaystyle\bigoplus_{i\in \mathbb{Z}} \h^{i}(\xx^e, X) \right\rangle}_{n-\grade(I,X)+1} ^{\Db(R)} .
\end{align*}

Since $\bigoplus_{i\in \mathbb{Z}} \h^{i}(\xx^e, X)\in \mod R$ is annihilated by $J$, hence it has a finite filtration by cyclic modules $R/\p$ where $\p\in \V(J)=\V(I)$. Thus, $\bigoplus_{i\in \mathbb{Z}} \h^{i}(\xx^e, X)\in \sw(\X)$. Hence, \begin{align*}
\k(\xx^e, X) \in 
 & {\left\langle \displaystyle\bigoplus_{i\in \mathbb{Z}} \h^{i}(\xx^e, X) \right\rangle}_{n-\grade(I,X)+1} ^{\sw(\X)}.
\end{align*}

Now for all prime ideals $\p$ of $R$, 
one has inequalities
\begin{align*}
\depth_{R_{\p}} X_{\p} = 
& \depth_{R_{\p}} (\Omega_R ^{m} M)_{\p}   \\
 \ge & \inf \{ \depth_{R_{\p}} \Omega_{R_{\p}} ^{m} M_{\p}, \, \depth R_{\p} \} \\
 \ge & \inf \{ \depth_{R_{\p}} M_{\p} +m, \, \depth R_{\p} \} \\
 \ge & \inf \{ \depth_{R_{\p}} M_{\p} + \Rfd_{R} M, \, \depth R_{\p} \} \\
 \ge & \depth R_{\p}.
\end{align*}
By taking the infimum of both sides of the above inequality as $\p$ runs through $\V(I)$, we have 
\begin{align*}
\grade(I,X) = 
& \inf \{\depth_{R_{\p}} X_{\p} \mid \p \in \V(I) \} \\
 \ge & \inf \{\depth R_{\p} \mid \p \in \V(I) \} \\
 = & \grade I .
\end{align*}

Thus 
$\k(\xx^e, X)$ belongs to 
$\langle {\bigoplus_{i\in \mathbb{Z}} \h^{i}(\xx^e, X)}\rangle_{\mu(I)-\grade I +1} ^{\sw(\X)}$. 
On the other hand, as $J=(x_1^e,...,x_n^e)\subseteq \ann_{R}\End_{\Dsg(R)}(X)$ by \cite[3.1 and Proposition 4.2]{dey2025annihilationcohomologystronggeneration}, hence 
$X$ is a direct summand of $X/\!\!/\xx = \k(\xx, X)$ in $\Dsg(R)$.
Therefore $X\in \langle {\bigoplus_{i\in \mathbb{Z}} \h^{i}(\xx^e, X)}\rangle_{\mu(I)-\grade I +1} ^{\sw(\X)}$.  Thus, recalling that $\sw(\X)=\langle \X\rangle$, and again the fact that every $R/J$-module is in $\filt_{\mod R}\{R/\p\mid \p\in \V(I)\}$, we get $\sw(\X)=\langle \filt_{\mod R}\{R/\p \mid \p\in \V(I)\} \rangle^{\sw(\X)}_{\mu(I)-\grade I+1}$. 

The proof of the last claim follows the same argument by replacing $\xx^e$ by $\xx$ and $J$ by $I$ everywhere. 
\end{proof}

We record the following remark in regards to the hypothesis of \Cref{mthm}.

\begin{rmk} 
\begin{enumerate}[\rm(1)]
    \item Let $\Phi$ be a closed subset of $\spec R$. Assume $\Phi\subseteq \Sing R$. Consider $\X:=\res\{R/\p\mid \p\in \Phi\}$. Then, $\ipd(\X)=\Phi$ (see \Cref{ipd}), and obviously $R/\p\in \X$ for all $\p\in \Phi$. Thus, with this $\X$ and a defining ideal $I$ of $\Phi$, the assumption of \Cref{mthm} is satisfied.

    \item Let $\X$ be a resolving subcategory of $\mod R$ and $\Phi$ a specialization-closed subset of $\spec R$ such that $\ipd(\X)\subseteq \Phi$ and $\X$ is dominant on $\Phi$ in the sense of \cite[Definition 4.1]{dom}. Then, $R/\p\in \sw(\X)$ for all $\p\in \Phi$. Indeed, let $r:=\Rfd(R/\p)$ and set $M:=\syz^r_R(R/\p)$. Then, $\operatorname{NF}(M)=\ipd(R/\p)\subseteq \V(\p) \subseteq \Phi$. 
    By hypothesis, $\X$ is dominant on  $\operatorname{NF}(M)$ and $\depth M_{\q}\ge \depth R_{\q}$ for all $\q\in \spec R$. Since $R\in \X$, we get $\syz^r_R(R/\p)=M\in \X$ by \cite[Theorem 4.3]{dom}. Thus, $R/\p\in \sw(\X)$.

    \item The condition $\ipd(\X)\subseteq \V(I)$ is weaker than $I\subseteq \ann_R \sw(\X)$, see \Cref{ipd}(4). 
\end{enumerate} 
\end{rmk}

By restricting \Cref{mthm} to the case where $\X$ is the resolving closure of the modules $R/\p$ with $\p\in\Sing R$, we obtain the following result.
\begin{cor}\label{main thm} Let $R$ be a commutative noetherian ring and 
$I$ an ideal of $R$ such that $\Sing R\subseteq \V(I)$. Then, $\operatorname{D}_{\operatorname{sg}}(R)=\langle \filt_{\mod R}\{R/\p \mid \p\in \V(I)\} \rangle_{\mu(I)-\grade I+1}$.  If, moreover, $I\subseteq \ann_R \Dsg(R)$. 
Then we have 
\begin{center}
$\Dsg(R) = 
{\langle \mod R/I 
\rangle}_{\mu(I)-\grade I+1}$.
\end{center}
In particular, we have the following inequalities when $I\subseteq \ann_R \Dsg(R)$.
\begin{align*}
\dim \Dsg(R) 
& \le (\radius_{\Dsg(R)} (\mod R/I) + 1)(\mu(I)-\grade I+1)-1 \\
& \le (\dim\Db(R/I) + 1)(\mu(I)-\grade I+1)-1.
\end{align*} 
\end{cor}

\begin{proof} In \Cref{mthm}, let $\X:=\res\{R/\p\mid \p\in \Sing R\}$. Then, we are done in view of $\sw(\X)=\Dsg(R)$ (see \cite[Corollary 4.3]{kos}) and  $\ipd(\X)= \Sing R$. 
\end{proof}

As a consequence of \Cref{main thm}, the finiteness of the dimension of the singularity category of $R$ can be summarized in the following form.
\begin{cor} \label{fin of dim}
The following conditions are equivalent.
\begin{enumerate}[\rm(a)]
\item 
The dimension of $\Dsg(R)$ is finite.
\item
For all ideals $I$ of $R$ contained in $\ann_R \Dsg(R)$, 
the radius of $\mod R/I$ in $\Dsg(R)$ is finite.
\item
The radius of $\mod (R/\ann_R \Dsg(R))$ in $\Dsg(R)$ is finite.
\item
There exists an ideal $I$ of $R$ contained in $\sqrt{\ann_R \Dsg(R)}$ such  that
the radius of $\mod R/I$ in $\Dsg(R)$ is finite.
\end{enumerate}
In particular, if there exists an ideal $I$ of $R$ contained in $\sqrt{\ann_R \Dsg(R)}$ such that the dimension of $\Db(R/I)$ is finite, then the dimension of $\Dsg(R)$ is finite.
\end{cor}
\begin{proof}
Since the inequality
$\dim \Dsg(R) \ge \radius_{\Dsg(R)}(\mod R/I)$
holds for all ideals $I$ of $R$, we get the implication (a)$\Rightarrow$(b). 
The implications (b)$\Rightarrow$(c)$\Rightarrow$(d) are trivial.
Assume that condition (d) holds. Then there exists an integer $e\ge0$ such that $I^{e}\subseteq\ann_R \Dsg(R)$. Combining \Cref{main thm} with the fact that $\mod R/I^{e}\subseteq \langle\mod R/I\rangle^{\Dsg(R)}_{e}$, we obtain (a).
\end{proof}

Now we give some more consequences of \Cref{mthm} that provide sufficient conditions for $\sw(\X)$ to possess a classical generator.
We denote by $\Min R$ the set of minimal prime ideals of $R$.

\begin{cor}\label{new1} Let $\X$ be a resolving subcategory of $\mod R$. Let $I$ be an ideal of $R$ such that $\ipd(\X)\subseteq \V(I)$.  Assume $R/\p\in \sw(\X)$ for every $\p\in V(I)$. If $\Dsg(R/\p)$ has classical generator for every $\p\in \Min(R/I)$, then $\sw(\X)$ has classical generator.  
\end{cor}

\begin{proof} It follows from hypothesis that $\Db( R/\p)$ has classical generator for every $\p\in \Min(R/I)$. It then follows from \cite[Theorem 1]{KS} and \cite[Lemma 2.4]{viet} that $\mod R/I$ has a classical generator, say $G$. Since $R/\p\in \mod R/I$ for every $\p \in \V(I)$, hence by \Cref{mthm}, \Cref{remsw} and \cite[Proposition 5.3]{DaoTakahashiradius}, $ G \oplus   I$ is a classical generator for $\sw(\X)$.  
\end{proof} 

In view of \cite[Theorem 1]{KS}, our next result \Cref{mincomp} refines \cite[Lemma 2.4]{viet}. 

\begin{cor}\label{mincomp} Let $I$ be an ideal of $R$ such that $\Sing R\subseteq \V(I)$. If $\Dsg(R/\p)$ has a classical generator for all $\p \in \Min(R/I)$, then $\Dsg(R)$ has a classical generator. In particular, if $R/I$ is a J-$2$ ring, then $\Dsg(R)$ has a classical generator. 
\end{cor}

\begin{proof} In \Cref{new1}, let $\X:=\res\{R/\p\mid \p\in \Sing R\}$. Then, we are done in view of $\sw(\X)=\Dsg(R)$ (see \cite[Corollary 4.3]{kos}) and  $\ipd(\X)= \Sing R$.
For the last part of the claim now, notice that the hypothesis implies $R/\p$ is J-$2$ ring for each $\p\in\Min(R/I)$, and thus each $\Dsg(R/\p)$ has a classical generator (see \cite[Proposition 2.8]{viet}) which concludes the proof. 
\end{proof}

As a consequence of \Cref{mincomp} we can refine \cite[Corollary 2.7]{viet}. We recall that by definition, $\operatorname{Reg}R$ is the complement of $\Sing R$ in $\spec(R)$.  

\begin{cor}\label{new3} Let $I$ be an ideal of $R$ such that $\Sing R\subseteq \V(I)$. If $\operatorname{Reg}(R/\p)$ contains a nonempty open subset  for every $\p\in \V(I)$, then $\Sing R$ is a closed subset. 
\end{cor}

\begin{proof} Since $\V(I)=\spec(R/I)$, we get, by \cite[Theorem 1.1]{viet}, $\Dsg(R/J)$ has a classical generator for every ideal $J\supseteq I$. Thus, $\Dsg(R)$ has a classical generator by \Cref{mincomp}, and so $\Sing R$ is closed by \cite[Lemma 2.6]{viet}. 
\end{proof}

The following result gives a sufficient condition for $\sw(\X)$ to be of finite dimension, in terms of strong generation in the module category.
\begin{cor}\label{new2} Let $\X$ be a resolving subcategory of $\mod R$. Let $I$ be an ideal of $R$ such that $I\subseteq \sqrt{\ann_R \sw(\X)}$.  Assume $R/\p\in \sw(\X)$ for every $\p\in \V(I)$.
If  for every $\p\in\Min(R/I)$, $\mod R/\p$ has a strong generator (in the sense of \cite[Section 4]{IyengarTakahashi2016}), then the dimension of $\sw(\X)$ is finite.
\end{cor}

\begin{proof} By hypothesis, $I^e\subseteq \ann_R \sw(\X)$ for some $e\ge 1$. Since $\sqrt I=\sqrt{I^e}$, so we may replace $I$ by $I^e$ and assume $I\subseteq \ann_R \sw(\X)$. By \Cref{mthm}, $\sw(\X)=\langle \mod R/I \rangle_r$ for some integer $r\ge 0$. Now \cite[Proposition 3.10]{dey2024stronggenerationmodulecategories} applied to the ring $R/I$ implies that there exists $G\in \mod R/I$ and integers $n,s\ge 0$ such that $\syz^s_{R/I}(\mod R/I)\subseteq |G|^{\mod R/I}_n$. Now in view of \cite[Proposition 5.3]{DaoTakahashiradius} and the fact that $G\oplus I\in \sw(\X)$, we conclude that $\sw(\X)=\langle G\oplus I\rangle_{rn(s+1)}$.
\end{proof}


We recall several terms for later use.
\begin{dfn} \label{def for lemma}
\begin{enumerate}[\rm(1)]
\item
We say that $R$ is {\em equicharacteristic} if $R$ contains a field as a subring.
\item
Assume that the Krull dimension of $R$ is finite.
We denote by $\Assh R$ the set of prime ideals $\p$ of $R$ such that the equality $\dim R = \dim R/\p$ holds. 
Note that the inclusion relation $\Min R \supseteq \Assh R$ holds. We say that $R$ is {\em equidimensional} if the two sets of prime ideals $\Min R$ and $\Assh R$ coincide.
\item
We say that $R$ is {\em half Cohen--Macaulay} if the inequality 
$2\depth R_{\p} \ge \dim R_{\p}$ holds for all prime ideals $\p$ of $R$. 
For example, Cohen--Macaulay ring is half Cohen--Macaulay.
As we will see in \Cref{dim 41}, a two-dimensional depth one noetherian local ring which are locally Cohen--Macaulay on the punctured spectrum is half Cohen--Macaulay but not Cohen--Macaulay.
\item 
Let $R$ be a finitely generated algebra over a field $k$ (resp. an equicharacteristic complete noetherian local ring with residue field $k$).
Then $R$ is isomorphic to 
\begin{center}
$k[X_1, \ldots,X_n]/(f_1, \ldots, f_c)$ (resp. 
$k[\![X_1, \ldots,X_n]\!]/(f_1, \ldots, f_c)$)
\end{center}
for some $n, c \in \mathbb{Z}_{\ge 0}$ and $f_{1}, \dots, f_{c}$ in $k[X_1, \ldots,X_n]$ (resp. $k[\![X_1, \ldots,X_n]\!]$).
The latter isomorphism follows from Cohen's structure theorem.
We set  $h = n - \dim R$. 
Note that $h$ equals to the height of the ideal $(f_1, \ldots, f_c)$ in $k[X_1, \ldots,X_n]$ (resp. $k[\![X_1, \ldots,X_n]\!]$).
We define the {\em Jacobian ideal} of $R$, denoted  by $\jac (R)$, as the ideal of $R $ generated by all $h \times h$ minors of the Jacobian matrix
$\partial(f_1,\dots,f_c)/\partial(X_1,\dots,X_n)=(\partial f_i/\partial X_j)_{i,j}$.
\item
Let $n$ be a nonnegative integer.
We define the {\em $n$-th cohomology  annihilator} of $R$, denoted by $\ca^{n}(R)$ as follows.
\begin{center}
$\ca^{n}(R)=
\displaystyle\bigcap_{m\ge n}
\left(
\displaystyle\bigcap_{M, N \in \mod R}
\ann_{R}\Ext_{R} ^{m}(M, N)
\right)$.
\end{center}
In addition, we define the {\em cohomology annihilator} of $R$ as follows.
\begin{center}
$\ca(R)=
\displaystyle\bigcup_{n\ge 0} \ca^{n}(R)$.
\end{center}
Note that the ascending chain of cohomology annihilators
$0=\ca^{0}(R) \subseteq \ca^{1}(R) \subseteq \ca^{2}(R) \subseteq \cdots$ stabilizes.
\end{enumerate}
\end{dfn}
Now, we state some ideals contained in the annihilator of singularity category.
Note that we can apply \Cref{main thm} to these ideals.
\begin{lem} \label{ideals in annihilator}
\begin{enumerate}[\rm(1)]
\item 
One has 
$\ca(R) \subseteq \ann_R \Dsg(R)$.
\item
Assume that $R$ is local. 
Then we have 
$\soc R \subseteq \ca(R)$.
\item
Let $R$ be a finitely generated algebra over a field, or an equicharacteristic complete local ring of dimension $d$. 
Assume that $R$ is equidimensional half Cohen--Macaulay.
Then one has 
$\jac (R) \subseteq \ca^{d+1}(R)$. 
In particular, we have
$\jac(R) \subseteq \ann_R \Dsg(R)$.
\end{enumerate}
\end{lem}
\begin{proof}
(1), (2), and (3) follow from \cite[Proposition 5.3(1)]{Liu}, \cite[Example 2.6]{IyengarTakahashi2016}, and \cite[Theorem 1.2]{IyengarTakahashi2021} respectively.
\end{proof}
In general, it may be difficult to determine the cohomology annihilator or the annihilator ideal of the singularity category of $R$.
However, we can easily compute the Jacobian ideal of $R$ whenever it is defined.
The following corollary suggests that under the assumptions of \Cref{ideals in annihilator}(3), we can calculate an upper bound for the dimension of the singularity category of $R$ by obtaining a radius of the category of modules over the factor ring of the Jacobian ideal of $R$.
\begin{cor} \label{for calculation}
Let $R$ be a finitely generated algebra over a field, or an equicharacteristic complete noetherian local ring of dimension $d$. 
If $R$ is equidimensional half Cohen--Macaulay, then one has
\begin{center}
$\Dsg(R) = 
{\langle \mod (R/\jac(R))
\rangle}_{\mu(\jac(R))-\grade (\jac(R))+1}$.
\end{center}
In particular, we have the following inequalities.
\begin{align*}
\dim \Dsg(R)
& \leq (\radius_{\Dsg(R)} (\mod (R/\jac(R)) + 1)(\mu(\jac(R))-\grade (\jac(R))+1)-1 \\
& \leq (\dim\Db(R/\jac(R)) + 1)(\mu(\jac(R))-\grade (\jac(R))+1)-1.
\end{align*}
\end{cor}
\begin{proof}
The assertions follow from \Cref{main thm} and \Cref{ideals in annihilator}.
\end{proof}
\begin{rmk} \label{depth zero case}
Let $(R,\m,k)$ be a noetherian local ring of depth zero.
Then
$\mu (\soc R) = \dim_{k} \soc R = \dim_{k}(0:_{R}\m)$ 
is the type of $R$, denoted by $\r(R)$.
In addition, the grade of $\soc R$ is zero.
Hence by \Cref{main thm}, one has
$\Dsg(R) 
= {\langle \mod (R/\soc R) \rangle}_{\r(R)+1}$.
However the following equality always holds.
\begin{center}
$\Dsg(R) 
= \langle \mod (R/\soc R) \rangle$. 
\end{center}
Indeed, for all finitely generated $R$-modules $M$,
the syzygies $\Omega_{R} ^{i} M$ $(i>0)$ are embedded into finite direct sums of $\m$ and they are annihilated by $\soc R$.
Combining this with \cite[Lemma 2.4(2)]{DaoTakahashi2015}, we obtain the above equality.
\end{rmk}
The following lemma is an immediate consequence of \cite[Corollary 3.3(2)]{AiharaTakahashi2015} and \cite[Proposition 7.37]{Rouquier}.
\begin{lem} \label{trivial der dim}
\begin{enumerate}[\rm(1)]
\item 
If $R$ is a regular ring with finite Krull dimension $d$, 
then one has 
$\Db(R) = 
{\langle R \rangle}_{d+1}$.
\item
If $R$ is an artinian ring, then one has
$\Db(R) = 
{\langle R/\rad R \rangle}_{\ell\ell(R)}$ 
where $\rad R$ is the Jacobson radical of $R$ and 
$\ell\ell(R) = \inf \{n\ge0 \mid (\rad R)^{n}=0\}$ 
is the Loewy length of $R$.
\end{enumerate}
\end{lem}
\begin{rmk} \label{dim vs rad}
Let $I$ be an ideal of $R$ contained in $\ann_{R}\Dsg(R)$.
As shown in \Cref{main thm}, it is not necessarily required to obtain upper bounds for the dimension of $\Db(R/I)$ so that we get that of $\Dsg(R)$.
It might be easier to obtain upper bounds for 
$\radius_{\Db(R/I)}(\mod R/I)$ than to get that of $\Db(R/I)$.

For example, suppose that there exist object $G$ in $\Db(R)$ and nonnegative integers $s,n$ 
such that $\Omega_{R}^{s}(\mod R) \subseteq {\langle G\rangle}_{n}^{\Db(R)}$ (e.g., the category $\mod R$ has a {\em strong generator} in the sense of Iyengar and Takahashi \cite[Section 4]{IyengarTakahashi2016}).
Then we can easily see that
$\mod R \subseteq {\langle
G\oplus R \rangle}_{n+s}^{\Db(R)}$
by taking projective resolutions of modules in $\mod R$; see also \cite[Lemma 2.9]{Oppermann}.
Hence if $R$ is a regular ring with finite Krull dimension $d$, then one has
$\mod R \subseteq 
{\langle R \rangle}_{d+1} ^{\Db(R)}$.
In particular, we have
$\radius_{\Db(R)}(\mod R) \le d$.
Moreover, if $R$ is a Cohen--Macaulay local ring of {\em finite CM type}, that is, the set of isomorphism classes of indecomposable maximal Cohen--Macaulay $R$-modules is finite,
then one has 
$\mod R \subseteq 
{\langle G \rangle}_{d+1} ^{\Db(R)}$
for some maximal Cohen--Macaulay $R$-module $G$.
These conclusions can be reached through a straightforward discussion. However, more generally, the dimensions of derived categories are already known; see \cite[Theorem 4.1]{AAITY}.
\end{rmk}
In the following example, we give an upper bound for the dimension of the singularity category over a complete noetherian local ring which is neither Cohen--Macaulay nor isolated singularity.
\begin{eg} \label{dim 41}
Let $k$ be a field with characteristic different from two,
and we consider 
\begin{center}
$R= k[\![X, Y, Z, W]\!]/(X^2, Y)\cap (Z, W) 
= k[\![X, Y, Z, W]\!]/(X^2 Z, X^2 W, YZ, YW)$. 
\end{center}
Since 
$(X^2 Z, X^2 W, YZ, YW) = (X^2, Y)\cap (Z, W)$ 
is a minimal primary decomposition, we have 
$\Ass R = \{ (x,y), (z,w)\} = \Min R =\Assh R$. 
Hence $R$ is a two-dimensional equicharacteristic equidimensional complete local ring. 
We can easily see that $x-z$ is a regular element of $R$ and
the ring $R/(x-z)$ has nonzero socle element $\overline{z}^2$.
Thus the depth of $R$ is one.

We shall show that $R$ is locally a hypersurface on the punctured spectrum but not an isolated singularity.
Let $\p$ be a prime ideal of $R$ different from 
$\m = (x,y,z,w)$.
If $\p$ does not contain $x$ or $y$, 
then $R_{\p}$ is isomorphic to a localization of a formal power series ring in two variables. 
Thus $R_{\p}$ is a regular local ring.
On the other hand, if $\p$ does not contain $z$ or $w$, 
then $R_{\p}$ is isomorphic to a localization of the ring obtained by factoring a formal power series ring in three variables by the square of a certain variable.
Hence $R_{\p}$ is a hypersurface but not a regular local ring.
Therefore $R$ is locally a hypersurface on the punctured spectrum. 
In particular, $R$ is half Cohen--Macaulay.

The height of the ideal $(X^2 Z, X^2 W, YZ, YW)$ is two in $k[\![X, Y, Z, W]\!]$. 
Hence the Jacobian ideal of $R$ is generated by $2 \times 2$ minors of Jacobian matrix of $R$, and we can compute that
\begin{center}
$\jac(R) = (x^{4}, x^{2}y, xz^{2}, xzw, xw^{2}, y^2)$.
\end{center}
Thus we have $\mu(\jac(R)) = 6$.
Since $\sqrt{\jac(R)} = (x,y)$ is a minimal prime ideal of $R$,  
we can see that 
$0 \le \grade(\jac(R)) \le \height(\jac(R))=0 $.

Let $S = R/\jac(R)$. 
We shall compute an upper bound for the dimension of the bounded derived category of $\mod S$.
Note that $S$ is isomorphic to the following ring.
\begin{center}
$k[\![X, Y, Z, W]\!]/
(X^{2}Z, X^{2}W, YZ, YW,X^{4}, X^{2}Y, XZ^{2}, XZW, XW^{2}, Y^{2})$.
\end{center} 
We denote by $\p = (x,y)S$ the unique minimal prime ideal of $S$.
For an object $A$ in $\Db(S)$, there exists an exact triangle
\begin{center}
$A/(0:_{S}\p)A \to A \to A/\p A \to \Sigma (A/(0:_{S}\p)A)$
\end{center}
in $\Db(S)$.
Thus we have 
$\Db(S) = 
\Db(S/(0:_{S}\p)) \ast \Db(S/\p)$.
We can see that 
$(0:_{S}\p) 
= (x^{3}, xy, xz, xw, z^{2}, zw, w^{2})S$, and hence one has
\begin{center}
$S/(0:_{S}\p) \cong k[\![X, Y, Z, W]\!]/
(X^{3}, XY, XZ, XW,
Y^{2}, YZ, YW,
Z^{2}, ZW,
W^{2})$.
\end{center}
Let $\n$ be the maximal ideal of $S/(0:_{S}\p)$. 
Then it holds that
$\n ^{3}=0 $ and $ 0\neq \overline{x}^{2} \in \n^{2}$ in  $S/(0:_{S}\p)$. 
Hence the Loewy length of $S/(0:_{S}\p)$ is three and we have
\begin{center}
$\Db(S/(0:_{S}\p)) 
={\langle k \rangle}_{3} ^{\Db(S/(0:_{S}\p)) } $.
\end{center}
Moreover, since
$S/\p \cong k[\![Z, W]\!]$
is a two-dimensional regular local ring, one has 
\begin{center}
$\Db(S/\p) =
{\langle S/\p \rangle}_{3} ^{\Db(S/\p)}$
\end{center}
by \Cref{trivial der dim}. 
Combining with the above results, we obtain
\begin{center}
$\Db(S) 
= {\langle k \oplus S/\p \rangle}_{6}$.
\end{center}
Therefore by using \Cref{main thm}, we have
\begin{center}
$\Dsg(R) 
={\langle k \oplus S/\p \rangle}_{(6-0+1) \cdot 6} 
={\langle k \oplus R/(x,y) \rangle}_{42}$.
\end{center}
Consequently, we obtain an upper bound for the dimension of singularity category of $R$ as follows:
\begin{center}
$\dim \Dsg(R) \le 41$.
\end{center}
\end{eg}
We apply \Cref{depth zero case} to compute an upper bound for the dimension of the singularity category.
\begin{eg} \label{eg of depth zero}
Let $k$ be a field
and $n>1$ an integer.
We consider 
\begin{center}
$R= k[\![X, Y]\!]/(X^{n}, X^{n-1}Y) 
= k[\![X, Y]\!]/(X^{n-1})\cap (X^{n}, Y)$. 
\end{center}
We can easily see that $R$ is a one-dimensional depth zero complete local ring with
$\spec R=\{(x), (x, y)\}=\Ass R \supsetneq \Min R=\{(x)\} = \Assh R$.
Note that $R$ is not an isolated singularity except for the case where $n=2$. 
We see that $\soc R = (x^{n-1})$, and
thus one has an isomorphism 
$R/\soc R \cong k[\![X, Y]\!]/(X^{n-1})$. 
This is a hypersurface.

Let $S=R/\soc R$ and $\p = xS \in \Min S$.
Then there are isomorphisms
$S/\p \cong k[\![Y]\!] \cong R/(x)$, 
and thus the equality
$\Db(S/\p)=
{\langle R/(x) \rangle}_{2}$
holds.
Now we shall calculate an upper bound for the dimension of $\Db(S)$.
Let $A$ be an object in $\Db(S)$.
Then there exists a filtration
\begin{center}
$0=\p^{n-1}A \subseteq \p^{n-2}A \subseteq \cdots
\subseteq \p A \subseteq A$
\end{center}
of bounded $R$-complexes.
For each $i=1, \ldots, n-1$, 
the complexes $\p^{i-1}A/\p^{i}A$ 
belongs to $\Db(S/\p)$.
Thus the inductive argument shows that 
$A$ belongs to 
${\langle \Db(S/\p) \rangle}_{n-1} ^{\Db(S)}
= {\langle R/(x) \rangle}_{2(n-1)} ^{\Db(S)}$.
Therefore we obtain 
$\Db(S)={\langle R/(x) \rangle}_{2(n-1)}$.
Hence by \Cref{depth zero case}, we conclude that 
\begin{center}
$\Dsg(R)={\langle R/(x) \rangle}_{2(n-1)}$.
\end{center}
In particular, we get an upper bound for the dimension of singularity category of $R$ as follows.
\begin{center}
$\dim \Dsg(R) \le 2(n-1)-1$.
\end{center}

When $n=2$, the generator $R/(x)$ of $\Dsg(R)$ coincides with the module $R/\soc R$, which 
is isomorphic to the residue field $k$ in $\Dsg(R)$ up to a shift. 
Hence one has 
$\Dsg(R)=
{\langle k \rangle}_{2}$
and we obtain an inequality
$\dim \Dsg(R) \le 1$.
\end{eg}
As a first step, we give an explicit generator and an upper bound for the dimension of the singularity category of $R$ when the dimension of $\Sing R$ is zero. 
\begin{cor} \label{cor of dimSing 0}
\begin{enumerate}[\rm(1)]
\item 
Suppose that the dimension of $\Dsg(R)$ is finite, 
and for all prime ideals $\p \in \spec R$ which is not a maximal ideal, 
the localization $R_{\p}$ is a regular local ring.
Then $\V(\ann_{R}\Dsg(R))$ consists of maximal ideals of $R$.
\item
Let $I$ be an ideal of $R$ contained in $\ann_{R}\Dsg(R)$. 
If $\V(I)$ consists of maximal ideals of $R$, then one has
\begin{center}
$\Dsg(R)=
{\left\langle \dfrac{R/I}{\rad(R/I)}
\right\rangle}_{\ell\ell(R/I)(\mu(I)-\grade I +1)}$.
\end{center}
In particular, we have an inequality
\begin{center}
$\dim \Dsg(R) \le 
\ell\ell(R/I)(\mu(I)-\grade I +1)-1$.
\end{center}
\end{enumerate}
\end{cor}
\begin{proof}
(1) By using \cite[Theorem 4.6]{Liu}, 
$\V(\ann_{R}\Dsg(R))$ coincides with the singular locus of $R$, which consists of maximal ideals of $R$ by the hypothesis.

(2) Since $R/I$ is an artinian ring, the equality 
$\Db(R/I) = {\langle (R/I)/\rad(R/I) 
\rangle}_{\ell\ell(R/I)}$ holds by \Cref{trivial der dim}.
Hence the assertion follows from \Cref{main thm}.
\end{proof}
The restriction to the case where $(R,\m,k)$ is a local ring, \Cref{cor of dimSing 0} recovers \cite[Lemma 6.5]{Liu}.
Combining this with \cite[Corollary 7.2]{IyengarTakahashi2016}, we obtain \cite[Theorem 1.3]{Liu} as follows.
\begin{cor}[Liu]\label{thm of Liu}
Let $(R,\m,k)$ be an equicharacteristic excellent local ring with an isolated singularity. 
Then the following hold.
\begin{enumerate}[\rm(1)]
\item
The ideal $\ann_{R}\Dsg(R)$ is $\m$-primary.
\item
For any $\m$-primary ideal $I$ contained in $\ann_{R}\Dsg(R)$, one has
$$
\Dsg(R)={\langle k\rangle}_{(\mu(I)-\depth R+1)\ell\ell(R/I)}.
$$
\end{enumerate}
\end{cor}
We denote by $\nil R$ the nilradical of $R$, and
set $\operatorname{n}(R) = \inf \{m\ge0\mid (\nil R)^{m}=0\}$.
We define $R_{\red} = R/\nil R$ as the associated reduced ring.
When $R$ is reduced, we denote by $\overline{R}$ the integral closure of $R$ in its total quotient ring $Q$, and denote by $C_{R} = (R:_{Q}\overline{R})$ the {\em conductor ideal} of $R$. 
See \cite[Section 31]{Matsumura} for the definition of a {\em Nagata ring}.
\begin{lem} \label{Nagata ring of dim1}
Let $R$ be a Nagata ring of dimension one.
Then one has
\begin{center}
$\Db(R)={\langle
\overline{R_{\red}} \oplus T/\rad T
\rangle}_{2\operatorname{n}(R)(\ell\ell(T)+1)}$,
\end{center}
where $\overline{R_{\red}}$ is the integral closure of  $R_{\red}$ in its total quotient ring, and $T$ is the quotient ring of $R_{\red}$ by its conductor ideal.
\end{lem}
\begin{proof}
By using \cite[Lemma 7.35]{Rouquier}, we may assume that $R$ is reduced, and that $\operatorname{n}(R)=1$.

Here we state some basic properties of $R$, which is a reduced Nagata ring of dimension one (cf. \cite[Proposition 5.6(1)]{IyengarTakahashi2016}). 
Let $\{\p_{1}, \ldots, \p_{m}\}$ be the set of minimal prime ideals of $R$.
Since $R$ is reduced, the integral closure $\overline{R}$ of $R$ in its total quotient ring is isomorphic to a direct product of $\overline{R/\p_{i}}$ by \cite[Corollary 2.1.13]{HunekeSwanson}, where the integral closure is considered in its field of fractions $\kappa(\p_{i})$ for each $i=1, \ldots, m$.
On the other hand, $\overline{R/\p_{i}}$ is integrally closed domain of dimension at most one for all $i=1, \ldots, m$. Hence the global dimension of $\overline{R}$ is at most one.
In addition, since $R$ is a Nagata ring, $\overline{R/\p_{i}}$ is a finitely generated $R$-module for all $i=1, \ldots, m$. Hence $\overline{R}$ is also finitely generated as an $R$-module.
Thus by \cite[Exercise 2.11]{HunekeSwanson}, the conductor $C_{R}$ has a nonzero divisor of $R$, and that $T=R/C_{R}$ has finite length. 

The proof of this lemma in the case where $R$ is a complete noetherian local ring of dimension one is given in 
\cite[Proposition 3.12]{AiharaTakahashi2015}.
Combining its proof with the facts for a reduced Nagata ring of dimension one stated above, the assertion holds.
\end{proof}
As a second step, we give an explicit generator and an upper bound for the dimension of the singularity category of $R$ in the case where the dimension of $\Sing R$ is one.

\begin{cor} \label{cor dimSing 1}
Let $I$ be an ideal of $R$ contained in $\ann_{R}\Dsg(R)$.
Suppose that $R/I$ is a quasi-excellent ring of finite Krull dimension.
Then the following hold.
\begin{enumerate}[\rm(1)]
\item 
The dimension of $\Dsg(R)$ is finite, and the ideal $\ann_{R}\Dsg(R)$ defines the singular locus of $R$.
\item
Assume that the dimension of $\Sing R$ is one
and $\V(I)=\Sing R$.
We set $S=R/I$ and 
denote by $T$ the quotient ring of $S_{\red}$ by its conductor ideal.
Then one has
\begin{center}
$\Dsg(R)= {\langle
\overline{S_{\red}}\oplus T/\rad T
\rangle}_{2\operatorname{n}(S)(\ell\ell(T)+1)
(\mu(I)-\grade I +1)}$.
\end{center}
In particular, we have
\begin{center}
$\dim \Dsg(R)
\le 2\operatorname{n}(S)(\ell\ell(T)+1)(\mu(I)-\grade I +1)-1$.
\end{center}
\end{enumerate}
\end{cor}
\begin{proof}
(1) Since $R/I$ is a quasi-excellent ring of finite Krull dimension, the dimension of $\Db(R/I)$ is finite by \cite[Main Theorem]{Aoki}.
Hence by \Cref{fin of dim}, the dimension of the singularity category of $R$ is also finite.
Thus the ideal $\ann_{R}\Dsg(R)$ defines the singular locus of $R$ by \cite[Theorem 4.6]{Liu}.

(2) Note that every quasi-excellent ring is a Nagata ring; see \cite[(33.H)]{Matsumura} for instance.
By the hypothesis, $S$ is a Nagata ring of dimension one.
Hence by \Cref{Nagata ring of dim1}, one has
\begin{center}
$\Db(S)={\langle
\overline{S_{\red}} \oplus T/\rad T
\rangle}_{2\operatorname{n}(S)(\ell\ell(T)+1)}$.
\end{center}
Therefore combining with this and \Cref{main thm}, the assertion holds.
\end{proof}
We give a simple example of the application of \Cref{cor dimSing 1}.
\begin{eg} \label{eg dimSing 1}
Let $k$ be a perfect field (so that we use the Jacobian criterion later).
We consider
\begin{center}
$R = k[X, Y, Z, W]/(XY, YZ, ZX)$ \quad
(or $k[\![X, Y, Z, W]\!]/(XY, YZ, ZX)$).
\end{center}
Then we have a prime decomposition 
$(XY, YZ, ZX) = (X,Y)\cap(Y,Z)\cap(Z,X)$.
Hence $R$ is a two-dimensional reduced excellent Cohen--Macaulay ring.
Thus $R$ satisfies the condition ($\mathsf{R}_{0}$), that is, $R_{\p}$ is a regular local ring for all minimal prime ideals $\p$ of $R$.
This implies that the dimension of the singular locus of $R$ is at most one.
Moreover, the prime ideal $\p =(x, y, z)$ of $R$ has height one and belongs to $\Sing R$. 
Therefore the dimension of $\Sing R$ is one.

Since $R$ is equidimensional and $k$ is perfect, the Jacobian ideal $\jac(R)$ defines the singular locus of $R$. 
Indeed, \cite[Corollary 16.20]{Eisenbud} yields the affine case and the local case follows from \cite[Theorem 15.17]{LW}.
Moreover, $\jac(R)$ is contained in $\ann_{R}\Dsg(R)$ by \Cref{ideals in annihilator}(3).
Therefore, we can take $\jac(R)$ as the ideal $I$ in \Cref{cor dimSing 1}.

Next, we consider the ring $S=R/\jac(R)$.
Since the height of the ideal 
$(XY, YZ, ZX)$ is two, we can easily compute that 
$\jac(R) = (x^{2}, y^{2}, z^{2})$.
Note that $\jac(R)$ is $\p$-primary and the equalities
$\mu(\jac(R))=3$, 
$\grade(\jac(R)) = \height(\jac(R)) = \height \p = 1$
hold.
Hence we have an isomorphism
\begin{center}
$S\cong k[X, Y, Z, W]/(X^{2}, Y^{2}, Z^{2}, XY, YZ, ZX)$ (or 
$k[\![X, Y, Z, W]\!]/(X^{2}, Y^{2}, Z^{2}, XY, YZ, ZX)$),
\end{center}
and $\nil S = (x, y, z)S$ is the unique minimal prime ideal of $S$.
Note that one has
$\operatorname{n}(S)=2$.
On the other hand, the associated reduced ring $S_{\red}$ is isomorphic to 
$R/\p \cong$
$k[W]$ (or $k[\![W]\!]$).
This is an integrally closed domain, and thus we have
$\overline{S_{\red}}=S_{\red}$, 
$C_{S_{\red}}=S_{\red}$, and
$T=S_{\red}/C_{S_{\red}}=0$.
Therefore by \Cref{cor dimSing 1}, one has
\begin{center}
$\Dsg(R)=
{\langle S_{\red}
\rangle}_{2\cdot 2(0+1)(3-1+1)}
= {\langle R/\p
\rangle}_{12}$.
\end{center}
In particular, we obtain an inequality
$\dim \Dsg(R) \le 11$.
\end{eg} 
\begin{rmk} \label{rmk for eg1 of dimSing 1}
\begin{enumerate}[\rm(1)]
\item 
In \Cref{eg dimSing 1}, $R$ is not a Gorenstein ring.
Indeed, the element $r=z-(x+y) \in R$ is not a zero divisor of $R$, and one has an isomorphism 
$R/(r) \cong k[X,Y]/(X^{2},XY,Y^{2})$
(or $k[\![X,Y]\!]/(X^{2},XY,Y^{2})$).
This is an artinian local ring with type two.
Hence $R$ is not a Gorenstein ring.
\item
In \Cref{eg dimSing 1}, we can also compute an upper bound for the dimension of $\Db(S)$ as follows.
Let $P=(x, y, z)S$ be the unique minimal prime ideal of $S$.
For all objects $A$ in $\Db(S)$, there exists an exact triangle
$\p A \to A \to A/\p A \to \Sigma(\p A)$ in $\Db(S)$.
Combining this with the fact that $P^{2}=0$, we have
$\Db(S) = {\langle
\Db(S/P) \rangle}_{2}
={\langle R/\p \rangle}_{4}$.
Hence by \Cref{for calculation}, we obtain the same conclusion 
$\Dsg(R)= {\langle R/\p
\rangle}_{12}$ 
as in \Cref{eg dimSing 1}.
\end{enumerate}
\end{rmk}
Recall that a local ring $R$ has {\em countable CM type} if the set of isomorphism classes of indecomposable maximal Cohen--Macaulay $R$-module is countable.
We apply \Cref{cor dimSing 1} to compute an upper bound for the dimension of the singularity category of a local hypersurface ring of uncountable CM type.
\begin{eg}\label{uncountable CM}
We consider
$R=\mathbb{C}[\![X, Y, Z]\!]/(X^{3}-Y^{4})$.
This is a two-dimensional complete hypersurface domain with the coefficient field $\mathbb{C}$ and has multiplicity three by \cite[Exercise 4.6.12]{BH}.
As we will see later, the singular locus of $R$ has dimension one.
However, $R$ is not countable CM type.
Indeed, assume that $R$ is countable CM type. 
Then $R$ is isomorphic to either 
\begin{center}
$\mathbb{C}[\![z_{0}, z_{1}, z_{2}]\!]/(z_{1}^{2}+z_{2}^{2})$
or
$\mathbb{C}[\![z_{0}, z_{1}, z_{2}]\!]/(z_{0}z_{1}^{2}+z_{2}^{2})$
\end{center}
by \cite[Theorem B]{BuchweitzGrerelSchreyer} or \cite[Theorem 14.16]{LW}.
In both cases, $R$ has multiplicity two by \cite[Exercise 4.6.12]{BH}.
This is a contradiction.
Thus $R$ is uncountable CM type.

Now, we shall compute an upper bound for the dimension of the singularity category of $R$ by using \Cref{cor dimSing 1}.
We can easily compute that 
$\jac(R)=(x^{2}, y^{3})$.
Note that the equalities
$\grade(\jac(R))=\height(\jac(R))=\height(\sqrt{\jac(R)})=\height(x,y)=1$, 
and
$\mu(\jac(R))=2$ hold.
By the Jacobian criterion \cite[Theorem 15.17]{LW}, we have
$\V(\jac(R))=\Sing R$.
Since the dimension of $R/\jac(R)$ is one, 
the dimension of the singular locus of $R$ is also one.
We set $S=R/\jac(R)$.
Then one has
$S_{\red} \cong R/(x, y) \cong \mathbb{C}[\![Z]\!]$,
and $(x,y)S$ is the unique minimal prime ideal of $S$.
Since one has
$(x,y)^{3}S \neq 0, (x,y)^{4}S = 0$ in $S$, 
we have 
$\operatorname{n}(S)=4$.
Hence by \Cref{ideals in annihilator}(3) and \Cref{cor dimSing 1}(2), we obtain 
\begin{center}
$\Dsg(R)={\langle
R/(x, y) \rangle}_{2\cdot4(0+1)(2-1+1)}
= {\langle R/(x, y) \rangle}_{16}$.
\end{center}
In particular, we have an inequality
\begin{center}
$\dim \Dsg(R) \le 15$.  
\end{center}
\end{eg}
We end this section by applying \Cref{cor dimSing 1} to a Cohen--Macaulay local ring $R$ with countable CM type. 
Under mild assumptions, the dimension of the singular locus of $R$ is at most one.
\begin{cor} \label{contable CM}
Let $(R,\m,k)$ be a Cohen--Macaulay local ring of countable CM type.
Let $I$ be an ideal of $R$ contained in $\ann_{R}\Dsg(R)$.
Assume that $R$ is complete, or $k$ is uncountable and $R/I$ is quasi-excellent.
Then the following hold.
\begin{enumerate}[\rm(1)]
\item
The ideal $\ann_{R}\Dsg(R)$ defines the singular locus of $R$.
\item
Suppose that the ideal $I$ defines the singular locus of $R$.
Then the inequality
\begin{align*}
\dim \Dsg(R) \le \min \lbrace
&(\ell\ell(R/I)+1)(\mu(I)-\dim R+1)-1,  \\
&2\operatorname{n}(S)(\ell\ell(T)+1)(\mu(I)-\grade I+1)-1 \rbrace
\end{align*}
holds, where $S=R/I$, and $T$ is the quotient ring of $S_{\red}$ by its conductor ideal.
\end{enumerate}
\begin{proof}
(1) The assertion follows from \Cref{cor dimSing 1}(1).

(2) Since $R$ is countable CM type and either $R$ is complete or $k$ is uncountable, the dimension of the singular locus of $R$ is at most one by \cite[Theorem 2.4]{Takahashi2007}.
When the ideal $I$ coincides with $R$, the singular locus of $R$ is empty, and that $\Dsg(R)$ is trivial. Thus we may assume that $I$ is a proper ideal.

If the dimension of $\Sing R$ is zero, then we have the following inequality by \cite[Lemma 6.5(2)]{Liu}.
\begin{center}
$\dim \Dsg(R) \le (\ell\ell(R/I)+1)(\mu(I)-\dim R+1)-1$.
\end{center}
In addition, since $I$ is an $\m$-primary ideal of $R$, 
we have an inequality
\begin{center}
$(\ell\ell(R/I)+1)(\mu(I)-\dim R+1)-1 
\le
2\operatorname{n}(S)(\ell\ell(T)+1)(\mu(I)-\grade I+1)-1$.
\end{center}
Note that 
$\operatorname{n}(S)=\ell\ell(R/I) \ge 1$, 
$\ell\ell(T)=0$,
and $\dim R=\grade I$ hold.
Hence we get the inequality as desired.

On the other hand, if the dimension of $\Sing R$ is one, then we obtain the following inequality by \Cref{cor dimSing 1}(2).
\begin{center}
$\dim \Dsg(R) \le 2\operatorname{n}(S)(\ell\ell(T)+1)(\mu(I)-\grade I+1)-1$.
\end{center}
Moreover, since the dimension of $R/I$ is one, its Loewy length is infinity.
Thus the following hold.
\begin{center}
$ \infty = (\ell\ell(R/I)+1)(\mu(I)-\dim R+1)-1 
>
2\operatorname{n}(S)(\ell\ell(T)+1)(\mu(I)-\grade I+1)-1$.
\end{center}
Hence, we arrived at the conclusion.
\end{proof}
\end{cor}

\begin{ac} Souvik Dey was partially supported by the
Charles University Research Center program No.UNCE/SCI/022 and a grant GA CR 23-05148S from the Czech Science Foundation.
Yuki Mifune would like to thank his supervisor Ryo Takahashi for giving many thoughtful questions and helpful discussions. He also thanks Yuya Otake and Kaito Kimura for their valuable comments. 
\end{ac}

\bibliographystyle{plain}
\bibliography{mainbib} 

\begin{thebibliography}{10}

\bibitem{AAITY}
Takuma Aihara, Tokuji Araya, Osamu Iyama, Ryo Takahashi, and Michio Yoshiwaki.
\newblock Dimensions of triangulated categories with respect to subcategories.
\newblock {\em J. Algebra}, 399:205--219, 2014.

\bibitem{AiharaTakahashi2015}
Takuma Aihara and Ryo Takahashi.
\newblock Generators and dimensions of derived categories of modules.
\newblock {\em Commun. Algebra}, 43(11):5003--5029, 2015.

\bibitem{Aoki}
Ko~Aoki.
\newblock Quasiexcellence implies strong generation.
\newblock {\em J. Reine Angew. Math.}, 780:133--138, 2021.

\bibitem{Avramov}
Luchezar~L. Avramov.
\newblock Infinite free resolutions.
\newblock In {\em Six lectures on commutative algebra. Lectures presented at the summer school, Bellaterra, Spain, July 16--26, 1996}, pages 1--118. Basel: Birkh{\"a}user, 1998.

\bibitem{AvramovFoxby1991}
Luchezar~L. Avramov and Hans-Bj{\o}rn Foxby.
\newblock Homological dimensions of unbounded complexes.
\newblock {\em J. Pure Appl. Algebra}, 71(2-3):129--155, 1991.

\bibitem{AvramovFoxby1997}
Luchezar~L. Avramov and Hans-Bj{\o}rn Foxby.
\newblock Ring homomorphisms and finite {Gorenstein} dimension.
\newblock {\em Proc. Lond. Math. Soc. (3)}, 75(2):241--270, 1997.

\bibitem{AvramovIyengarLipman}
Luchezar~L. Avramov, Srikanth~B. Iyengar, and Joseph Lipman.
\newblock Reflexivity and rigidity for complexes, {I}: {Commutative} rings.
\newblock {\em Algebra Number Theory}, 4(1):47--86, 2010.

\bibitem{bah2}
Abdolnaser Bahlekeh, Shokrollah Salarian, Ryo Takahashi, and Zahra Toosi.
\newblock Spanier-{Whitehead} categories of resolving subcategories and comparison with singularity categories.
\newblock {\em Algebr. Represent. Theory}, 25(3):595--613, 2022.

\bibitem{BallardFaveroKatzarkov}
Matthew Ballard, David Favero, and Ludmil Katzarkov.
\newblock Orlov spectra: bounds and gaps.
\newblock {\em Invent. Math.}, 189(2):359--430, 2012.

\bibitem{BensonIyengarKrause}
Dave Benson, Srikanth~B. Iyengar, and Henning Krause.
\newblock A local-global principle for small triangulated categories.
\newblock {\em Math. Proc. Camb. Philos. Soc.}, 158(3):451--476, 2015.

\bibitem{BH}
Winfried Bruns and J\"{u}rgen Herzog.
\newblock {\em Cohen-{M}acaulay rings}, volume~39 of {\em Cambridge Studies in Advanced Mathematics}.
\newblock Cambridge University Press, Cambridge, 1998.

\bibitem{BuchweitzGrerelSchreyer}
R.-O. Buchweitz, G.-M. Greuel, and F.-O. Schreyer.
\newblock Cohen-{Macaulay} modules on hypersurface singularities. {II}.
\newblock {\em Invent. Math.}, 88:165--182, 1987.

\bibitem{B}
Ragnar-Olaf Buchweitz.
\newblock {\em Maximal {Cohen}-{Macaulay} modules and {Tate} cohomology. {With} appendices by {Luchezar} {L}. {Avramov}, {Benjamin} {Briggs}, {Srikanth} {B}. {Iyengar} and {Janina} {C}. {Letz}}, volume 262 of {\em Math. Surv. Monogr.}
\newblock Providence, RI: American Mathematical Society (AMS), 2021.

\bibitem{ChristensenFoxbyFrankild}
Lars~Winther Christensen, Hans-Bj{\o}rn Foxby, and Anders Frankild.
\newblock Restricted homological dimensions and {Cohen}-{Macaulayness}.
\newblock {\em J. Algebra}, 251(1):479--502, 2002.

\bibitem{DaoTakahashiradius}
Hailong Dao and Ryo Takahashi.
\newblock The radius of a subcategory of modules.
\newblock {\em Algebra Number Theory}, 8(1):141--172, 2014.

\bibitem{DaoTakahashi2015}
Hailong Dao and Ryo Takahashi.
\newblock Upper bounds for dimensions of singularity categories.
\newblock {\em C. R., Math., Acad. Sci. Paris}, 353(4):297--301, 2015.

\bibitem{dey2024stronggenerationmodulecategories}
Souvik Dey, Pat Lank, and Ryo Takahashi.
\newblock Strong generation for module categories, https://arxiv.org/abs/2307.13675, 2024.

\bibitem{dey2025annihilationcohomologystronggeneration}
Souvik Dey, Jian Liu, Yuki Mifune, and Yuya Otake.
\newblock Annihilation of cohomology and (strong) generation of singularity categories, https://arxiv.org/abs/2503.24186, 2025.

\bibitem{Eisenbud}
David Eisenbud.
\newblock {\em Commutative algebra. {With} a view toward algebraic geometry}, volume 150 of {\em Grad. Texts Math.}
\newblock Berlin: Springer-Verlag, 1995.

\bibitem{IyengarTakahashi2016}
Srikanth~B. Iyengar and Ryo Takahashi.
\newblock Annihilation of cohomology and strong generation of module categories.
\newblock {\em Int. Math. Res. Not.}, 2016(2):499--535, 2016.

\bibitem{viet}
Srikanth~B. Iyengar and Ryo Takahashi.
\newblock Openness of the regular locus and generators for module categories.
\newblock {\em Acta Math. Vietnam.}, 44(1):207--212, 2019.

\bibitem{IyengarTakahashi2021}
Srikanth~B. Iyengar and Ryo Takahashi.
\newblock The {Jacobian} ideal of a commutative ring and annihilators of cohomology.
\newblock {\em J. Algebra}, 571:280--296, 2021.

\bibitem{KS}
Henning Krause and Greg Stevenson.
\newblock A note on thick subcategories of stable derived categories.
\newblock {\em Nagoya Math. J.}, 212:87--96, 2013.

\bibitem{LW}
Graham~J. Leuschke and Roger Wiegand.
\newblock {\em Cohen-{M}acaulay representations}, volume 181 of {\em Mathematical Surveys and Monographs}.
\newblock American Mathematical Society, Providence, RI, 2012.

\bibitem{Liu}
Jian Liu.
\newblock Annihilators and dimensions of the singularity category.
\newblock {\em Nagoya Math. J.}, 250:533--548, 2023.

\bibitem{Matsui}
Hiroki Matsui.
\newblock Prime thick subcategories and spectra of derived and singularity categories of {Noetherian} schemes.
\newblock {\em Pac. J. Math.}, 313(2):433--457, 2021.

\bibitem{Matsumura}
Hideyuki Matsumura.
\newblock {\em Commutative algebra. 2nd ed}, volume~56 of {\em Math. Lect. Note Ser.}
\newblock The Benjamin/Cummings Publishing Company, Reading, MA, 1980.

\bibitem{Oppermann}
Steffen Oppermann.
\newblock Lower bounds for {Auslander}'s representation dimension.
\newblock {\em Duke Math. J.}, 148(2):211--249, 2009.

\bibitem{Rouquier}
Rapha\"el Rouquier.
\newblock Dimensions of triangulated categories.
\newblock {\em J. K-Theory}, 1(2):193--256, 2008.

\bibitem{HunekeSwanson}
Irena Swanson and Craig Huneke.
\newblock {\em Integral closure of ideals, rings, and modules}, volume 336 of {\em Lond. Math. Soc. Lect. Note Ser.}
\newblock Cambridge: Cambridge University Press, 2006.

\bibitem{Takahashi2007}
Ryo Takahashi.
\newblock An uncountably infinite number of indecomposable totally reflexive modules.
\newblock {\em Nagoya Math. J.}, 187:35--48, 2007.

\bibitem{stcm}
Ryo Takahashi.
\newblock Classifying thick subcategories of the stable category of {Cohen}-{Macaulay} modules.
\newblock {\em Adv. Math.}, 225(4):2076--2116, 2010.

\bibitem{kos}
Ryo Takahashi.
\newblock Reconstruction from {Koszul} homology and applications to module and derived categories.
\newblock {\em Pac. J. Math.}, 268(1):231--248, 2014.

\bibitem{dom}
Ryo Takahashi.
\newblock Classification of dominant resolving subcategories by moderate functions.
\newblock {\em Ill. J. Math.}, 65(3):597--618, 2021.

\end{thebibliography}

\end{document}